\documentclass{aimsbis}

\usepackage[all]{xy} %commutative diagrams

\usepackage{pstricks,pst-plot}
\usepackage{amsmath}
\usepackage{paralist}
\usepackage{graphics} %% add this and next lines if pictures should be in esp format    
\usepackage{epsfig} %For pictures: screened artwork should be set up with an 85 or 100 line screen
\usepackage{color,subfigure}
 \usepackage[colorlinks=true]{hyperref}
\hypersetup{urlcolor=blue, citecolor=red}
\usepackage[capitalise]{cleveref} % must be loaded after hyperref

\theoremstyle{plain}
    \newtheorem{theorem}{Theorem}
    \newtheorem{corollary}[theorem]{Corollary}
    \newtheorem{proposition}{Proposition}[section]
    \newtheorem{corollarysection}{Corollary}[section]
    \newtheorem{lemma}[proposition]{Lemma}

\newtheorem*{strangerlemma}{Lemma}

\newtheorem{claim}{Claim}
\newtheorem*{othertheorem}{Theorem}

\theoremstyle{definition}
    \newtheorem{definition}[proposition]{Definition}

\theoremstyle{remark}
    \newtheorem{remark}[proposition]{Remark}

\def\AA{{\mathbb A}}

 \def\NN{{\mathbb N}}  
 \def\RR{{\mathbb R}} \def\SS{{\mathbb S}} 
   
 \def\ZZ{{\mathbb Z}}

\def\Si{\Sigma}

\def\cA{{\mathcal A}}  \def\cG{{\mathcal G}}  \def\cS{{\mathcal S}}
\def\cB{{\mathcal B}}    
\def\cC{{\mathcal C}}   \def\cO{{\mathcal O}} \def\cU{{\mathcal U}}
\def\cD{{\mathcal D}}   \def\cP{{\mathcal P}} \def\cV{{\mathcal V}}
\def\cE{{\mathcal E}}   \def\cQ{{\mathcal Q}} \def\cW{{\mathcal W}}
\def\cF{{\mathcal F}}  \def\cL{{\mathcal L}} \def\cR{{\mathcal R}}

\def\Diff{\operatorname{Diff}}

\def\dist{\operatorname{dist}}

\def\Orb{\operatorname{Orb}}

\def\Int{\operatorname{int}}
\def\cl{\operatorname{cl}}
\def\Id{\operatorname{Id}}
\def\size{\operatorname{size}}

\title[A Franks Lemma that Preserves Invariant Manifolds]
      {A Franks Lemma that Preserves Invariant Manifolds}

\author[Nikolaz Gourmelon]{}


 \email{ngourmel@math.u-bordeaux1.fr}

\thanks{The author was supported by the Institut de Math\'ematiques de Bourgogne (Universit\'e de Dijon) by IMPA (Rio de Janeiro) and by the Institut de Math\'ematiques de Bordeaux (Universit\'e Bordeaux I)}

\begin{document}
\maketitle

% Enter the first author's name and address:
\centerline{\scshape Nikolaz Gourmelon}
\medskip
{\footnotesize
% please put the address of the first author
 \centerline{Institut de Math\'ematiques de Bordeaux}
 \centerline{Universit\'e Bordeaux 1}
   \centerline{351, cours de la Lib\'eration}
   \centerline{ F 33405 TALENCE cedex, FRANCE}
} % Do not forget to end the {\footnotesize by the sign }

\bigskip

% The name of the associate editor will be entered by an editorial staff
% \centerline{(Communicated by the associate editor name)}

%The abstract of your paper
\renewcommand{\abstractname}{R\'esum\'e}
\begin{abstract} 
D'apr\`es un c\'el\`ebre lemme de John Franks, toute perturbation de la diff\'erentielle d'un diff\'eomorphisme $f$ le long d'une orbite p\'eriodique est r\'ealis\'ee par une $C^1$-perturbation $g$ du diff\'eomorphisme sur un petit voisinage de ladite orbite. On n'a cependant aucune information sur le comportement des vari\'et\'es invariantes de l'orbite p\'eriodique apr\`es perturbation.

Nous montrons que si la perturbation de la d\'eriv\'ee est obtenue par une isotopie le long de laquelle existent les vari\'et\'es stables/instables fortes de certaines dimensions, alors on peut faire la perturbation ci-dessus en pr\'eservant les vari\'et\'es stables/instables semi-locales correspondantes. Ce r\'esultat a de nombreuses applications en syst\`emes dynamiques de classe $C^1$. Nous en d\'emontrons quelques unes. \end{abstract}
\renewcommand{\abstractname}{Abstract}

\begin{abstract}A well-known lemma by John Franks asserts that one obtains any perturbation of the derivative of a diffeomorphism along a periodic orbit by a $C^1$-perturbation of the whole diffeomorphism on a small neighbourhood of the orbit. However, one does not control where the invariant manifolds of the orbit are, after perturbation.

We show that if the perturbated derivative is obtained by an isotopy along which some strong stable/unstable manifolds  of some dimensions exist, then the Franks perturbation can be done preserving the corresponding stable/unstable semi-local manifolds. This is a general perturbative tool in $C^1$-dynamics that has many consequences. We give simple examples of such consequences, for instance a generic dichotomy between dominated splitting and small stable/unstable angles inside homoclinic classes.

\end{abstract}

\maketitle

\section{Introduction}

A few $C^1$-specific tools and ideas are fundamental in the study of the dynamics of $C^1$-generic diffeomophisms, that is, diffeomorphisms of a residual subset of the set $\Diff^1(M)$ of $C^1$-diffeomorphisms on a Riemannian manifold $M$.

On the one hand, one relies on closing and connecting lemmas to create periodic points and to create homoclinic relations between them.
The $C^1$-Closing Lemma of Pugh~\cite{Pu67} allows to close a recurrent orbit by an arbitrarily small $C^1$-perturbation. 
The connecting lemma of Hayashi~\cite{Hay}, whose proof relies on ideas derived from that of the closing lemma, says that if the unstable manifold of a saddle point accumulates on a point of the stable manifold of another saddle, then a $C^1$-perturbation creates a transverse intersection between the two manifolds. That was further generalized by Wen, Xia and Arnaud in~\cite{WX,arn1} and Bonatti and Crovisier in~\cite{BoCro,Cr1}, where powerful generic consequences are obtained.

On the other hand, we have tools to create dynamical patterns by $C^1$-perturbations in small neighbourhoods of periodic orbit.
John Franks~\cite{Fr} introduced a lemma that allows to reach any perturbation of the derivative along a periodic orbit as a $C^1$-perturbation of the whole diffeomorphism on an arbitrarily small neighbourhood of that orbit. This allows to systematically reduce $C^1$-perturbations along periodic orbits to linear algebra. 
%See for instance ~\cite{M1} 

Other perturbation results are about generating homoclinic tangencies by $C^1$-perturbations near periodic saddle points. To prove the Palis $C^1$-density conjecture in dimension $2$, that is, to prove that there is a $C^1$-dense subset of diffeomorphisms of surfaces that are hyperbolic  or admit a homoclinic tangency, Pujals and Sambarino~\cite{PuSam} first show that if the dominated splitting between the stable and unstable directions of a saddle point is not strong enough, then a $C^1$-perturbation of the derivative along the orbit induces a small angle between the two eigendirections. They apply the Franks' Lemma and do another perturbation to obtain a tangency between the two manifolds. In~\cite{W1}, Wen gave a generalization of that first step in dimension greater than $2$ under similar non-domination hypotheses.

These perturbations results rely on the Franks' lemma which unfortunately fails to yield any information on the behaviour of the invariant manifolds of the periodic point after perturbation. In particular, one does not control a priori what homoclinic class the periodic point will belong to, what strong connections there may be after perturbation, and it may not be possible to apply a connecting lemma in order to recreate a broken homoclinic relation.

\medskip
In~\cite{Gou}, a technique is found to preserve any fixed finite set in the invariant manifolds of a periodic point for particular types of perturbations along a periodic orbit. In particular it implies that one can create homoclinic tangencies inside homoclinic classes on which there is no stable/unstable uniform dominated splitting. This technique however is complex and difficult to adapt to other contexts. 

In this paper, we provide a simple setting in which the Franks' perturbation lemma can be tamed into preserving most of the invariant manifolds of the saddle point. Let us first state the Franks' Lemma:
\begin{strangerlemma}[Franks]
Let $f$ be a diffeomorphism. For all $\epsilon>0$, there is $\delta>0$ such that, for any periodic point $P$ of $f$, for any $\delta$-perturbation $(B_1,...,B_p)$ of the tuple $(A_1,...,A_p)$ of matrices that corresponds to the derivative $Df$ along the orbit $\Orb_P$ of $P$, for any neighbourhood $U$ of $\Orb_P$, one finds a $C^1$ $\epsilon$-perturbation $g$ of $f$ on $\cU$ that coincides with $f$ throughout $\Orb_P$ and whose derivative along it corresponds to $(B_1,...,B_p)$.
\end{strangerlemma}

We introduce a perturbation theorem that extends the Franks' Lemma, controlling both the behaviour of the invariant manifolds of $\Orb_P$, and the size of the $C^1$-perturbation needed to obtain the derivative $(B_1,...,B_p)$. Precisely, we prove that if the perturbation is done by an isotopy along a path of 'acceptable derivatives', that is, if the strong stable/unstable directions of some indices exist all along that path, then the diffeomorphism $g$ can be chosen so that it preserves the corresponding local strong stable/unstable manifolds outside of an arbitrarily small neighbourhood. Moreover, the size of the perturbation can be found arbitrarily close to the radius of the path. 

In order to prove our main theorem, we will rely on the fundamental $C^r$-perturbative Proposition~\ref{p.pertpropsimple} and the $C^1$-linearization \cref{c.linearisationC1}. These results are stated in \cref{s.mainpertpropos}.  
In \cref{s.isotopic}, we show that \cref{p.pertpropsimple} and its corollary induce the main theorem.  \cref{p.pertpropsimple} is proved in \cref{s.prop4}.

In section~\ref{s.consequences}, we give examples of a few isotopic perturbative results on linear cocycles, to show possible applications of our main theorem. For instance, we can turn the eigenvalues of a large period saddle point to have real eigenvalues, and preserve at the same time most of its strong stable/unstable manifolds. We also deduce a generic dichotomy inside homoclinic classes between dominated splittings and small angles. 

This result has already allowed a number of new developments by Potrie~\cite{Po} and Bonatti, Crovisier, D\'iaz and Gourmelon~\cite{BCDG}). Some  impressive results have recently been announced by Bonatti and Shinohara, and by Bonatti, Crovisier and Shinohara. These are detailed in the next section.

\bigskip

\noindent{\em {\bf Remerciements :}
Je remercie chaleureusement Jairo Bochi, Christian Bonatti, Sylvain Crovisier, Lorenzo D\'iaz et Rafael Potrie pour de nombreuses discussions, suggestions et encouragements ainsi que Marcelo Viana, le CNPQ et l'IMPA (Rio de Janeiro).} 
\medskip

\subsection{Statement of results}\label{s.statementresults}
Let $A$ be a linear map such that its eigenvalues $\lambda_1,\ldots,\lambda_d$, counted with multiplicity and ordered by increasing moduli, satisfy $|\lambda_i|<\min(|\lambda_{i+1}|,1)$. Then the {\em $i$-strong stable direction} of $A$ is defined as the $i$-dimensional invariant space corresponding to eigenvalues $\lambda_1,...,\lambda_i$.

If $P$ is a periodic point of period $p$ for a diffeomorphism $f$ and if the first return map $Df^p$ admits an $i$-strong stable direction, then there is inside the stable set of the orbit $\Orb_P$ of $P$ a unique boundaryless $i$-dimensional immersed $f$-invariant manifold that is tangent to that direction at $P$. We call it the {\em $i$-strong stable manifold} of $\Orb_P$ for $f$, and denote it by  $W^{i,ss}(P,f)$. One defines symmetrically the {\em $i$-strong unstable manifolds}, replacing $f$ by $f^{-1}$, and denote them by $W^{i,uu}(P,f)$. 
We denote by $W^{s/u}(P,f)$ the {\em stable/unstable} manifold of the orbit $\Orb_P$ of $P$, that is the strong stable/unstable manifold of maximum dimension.

\begin{definition}
A set $W^{+}(f,P)$ is a {\em local stable manifold} of $\Orb_P$ for $f$ if 
it is a strictly $f$-invariant\footnote{That is, $f\bigl[W^+(f,P)\bigr]$ is in the interior $\Int W^+(f,P)$ of $W^+(f,P)$, where $\Int W^+(f,P)=W^+(f,P)\setminus \partial W^+(f,P)$.} union of disjoint disks $\{D_{n}\}_{0\leq n<p}$, 
where each $D_{n}$ is a smooth ball inside the strong stable manifold $W^{ss,i}(f)$ that contains $f^n(P)$ in its interior.
\end{definition}

Symmetrically, a set $W^{-}(f,P)$ is a {\em local unstable manifold} of $\Orb_P$ for $f$ if it is a local stable manifold of $\Orb_P$ for $f^{-1}$.

Finally, if we have both $f=g$ and $f^{-1}=g^{-1}$ by restriction to (resp. outside) some set $K$, then we write "$f^{\pm 1}=g^{\pm 1}$ on $K$" (resp.  "$f^{\pm 1}=g^{\pm 1}$ outside $K$").

We are now ready to state the main theorem:

\begin{theorem}\label{t.mainsimplestatement}
Let $P$ be a $p$-periodic point for a diffeomorphism $f$ on a Riemannian manifold $(M,\|.\|)$. Fix a path $$\{\cA_{t}=(A_{1,t},\ldots,A_{p,t})\}_{t\in [0,1]}$$ where each $A_{n,t}$ is a linear map from $T_{f^{n-1}(P)}M$ to $T_{f^{n}(P)}M$, and the $p$-tuple $\cA_0=(A_{1,0},\ldots,A_{p,0})$ is the derivative of $f$ along $\Orb_P$. Then, 
\begin{itemize}
\item for any $\delta$ greater than the radius of the path $\cA_t$, that is,
$$\delta>\max_{1\leq n\leq p\atop t\in[0,1]}\left\{\|A_{n,t}-A_{n,0}\|, \|A^{-1}_{n,t}-A^{-1}_{n,0}\|\right\}\footnote{$\|A\|$ is the operator norm of the morphism of Euclidean spaces $A\colon T_{f^{n-1}(P)}M \to T_{f^{n}(P)}M$.},$$
\item for any neighborhood $U$ of $\Orb_P$,
\item for any local stable and unstable manifolds $W^+(f,P)$ and $W^-(f,P)$,
\end{itemize}
there is a $\delta$-perturbation $g$ of $f$ for the $C^1$-topology, and a pair of local stable and unstable manifolds $W^+(g,P)$ and $W^-(g,P)$, such that it holds: 
\begin{itemize}
\item $f^{\pm 1}=g^{\pm 1}$ throughout $\Orb_P$ and outside $U$,
\item the derivative of $g$ along $\Orb_P$ is the tuple $\cA_1=(A_{1,1},\ldots,A_{p,1})$,
\item For each $i\in \NN$, if the linear endomorphism $B_t=A_{p,t}\circ ... \circ A_{1,t}$ admits an $i$-strong stable (resp. unstable) direction for all $t\in[0,1]$, we have
$$\bigl[W^{ss,i}(g,P)\cap W^+(g,P)\bigr]\setminus U=\bigl[W^{+,i}(f,P)\cap W^+(f,P)\bigr]\setminus U,$$
and likewise for the strong unstable manifolds.
\end{itemize}
\end{theorem}

That is, one can do the Franks' perturbation while preserving the "semilocal" $i$-strong stable manifold, whenever the $i$-strong stable direction exists all along the isotopy by which we perturb the derivative, and likewise for strong unstable manifolds.

For $\theta\in\{s,u\}$ or any  $\theta$ of the form $"i,ss"$ or $"i,uu"$ we denote by $W^\theta_\varrho(P,f)$ be  the set of points in $W^\theta(P,f)$ whose distance to the orbit of $x$ within the immersed manifold $W^\theta(P,f)$ is strictly  less than $\varrho$.  We call it  the ($i$-strong) (un)stable manifold {\em of size $\varrho$}.
We can state a slightly stronger version of \cref{t.mainsimplestatement}:

\begin{theorem}\label{t.mainsimplestatementmoregeneral}
Let $P$ be a $p$-periodic point for a diffeomorphism $f$ on a Riemannian manifold $(M,\|.\|)$. Fix a path $$\{\cA_{t}=(A_{1,t},\ldots,A_{p,t})\}_{t\in [0,1]}$$ where each $A_{n,t}$ is a linear map from $T_{f^{n-1}(P)}M$ to $T_{f^{n}(P)}M$, and the $p$-tuple $\cA_0=(A_{1,0},\ldots,A_{p,0})$ is the derivative of $f$ along $\Orb_P$. Let $I$ (resp. $J$) be the set of integers $i>0$ such that, for all $t\in[0,1]$, the linear endomorphism $B_t=A_{p,t}\circ ... \circ A_{1,t}$ admits an $i$-strong stable (resp. unstable) direction. Then, 
\begin{itemize}
\item for any $\delta$ greater than the radius of the path $\cA_t$,
\item for any neighborhood $U$ of $\Orb_P$ and for any $\varrho>0$,
\item for any families of compact sets 
\begin{align*}
K_i&\subset W^{i,ss}_\varrho(P,f)\setminus U\\
L_j&\subset W^{j,uu}_\varrho(P,f)\setminus U,
\end{align*}
 where $i\in I$ and $j\in J$,  
\end{itemize}
there is a $\delta$-perturbation $g$ of $f$, for the $C^1$-topology, such that it holds: 
\begin{itemize}
\item $f^{\pm 1}=g^{\pm 1}$ throughout $\Orb_P$ and outside $U$,
\item the derivative of $g$ along $\Orb_P$ is the tuple $\cA_1=(A_{1,1},\ldots,A_{p,1})$,
\item  For all $(i,j)\in I\times J$, we have
 $$K_i\subset W^{i,ss}_\varrho(P,g)\mbox{ and }L_j\subset W^{j,uu}_\varrho(P,g).$$
\end{itemize}
\end{theorem}

%You can rescale the whole picture (to 80% for instance) by using the command \def\JPicScale{0.8}
\begin{figure}[hbt] \label{f.figlocal}
\ifx\JPicScale\undefined\def\JPicScale{1}\fi
\psset{unit=\JPicScale mm}
\psset{linewidth=0.2,dotsep=1,hatchwidth=0.3,hatchsep=1.5,shadowsize=1,dimen=middle}
\psset{dotsize=0.7 2.5,dotscale=1 1,fillcolor=black}
\psset{arrowsize=2 2,arrowlength=1,arrowinset=0.25,tbarsize=0.7 5,bracketlength=0.15,rbracketlength=0.15}
\psset{xunit=.5pt,yunit=.5pt,runit=.5pt}
\begin{pspicture}(0,250)(600,600)
%vari?t? stable
\pspolygon[](120,480)(580,480)(480,300)(20,300)
\rput(75,315){$W^s(P,f)$}
%domaine K
{\pscustom[linewidth=1,linecolor=black,fillstyle=solid,fillcolor=gray,opacity=1]
{\newpath \moveto(370,450)
\curveto(320,462)(208,452)(160,440)\curveto(120,430)(100,390)(120,370)
\curveto(150,340)(337,306)(450,350)\curveto(580,400)(410,440)(370,450)
\closepath }}
\rput(210,360){\large $K_2$}
%le trou dans K
{\pscustom[linewidth=1,linecolor=black,fillstyle=solid,fillcolor=white]
{\newpath
\moveto(280,410)\curveto(258,402)(253,386)(260,380)
\curveto(270,370)(360,380)(360,400)\curveto(360,426)(296,415)(280,410)
\closepath}}

%{\pscustom[linewidth=1,linecolor=black,fillstyle=solid,fillcolor=gray,opacity=1]
%{\newpath \moveto(100,400)
%\curveto(150,490)(570,490)(520,400)
%\curveto(470,310)(50,310)(100,400)
%\closepath }}
%%{\newpath \moveto(370,450)
%%\curveto(320,462)(208,452)(160,440)\curveto(120,430)(100,390)(120,370)
%%\curveto(150,340)(337,306)(450,350)\curveto(580,400)(410,440)(370,450)
%%\closepath }}
%\rput(210,360){\small $W^s_\varrho(x,f)\setminus V$}
%%le trou dans K
%{\pscustom[linewidth=1,linecolor=black,fillstyle=solid,fillcolor=white]
%{\newpath
%\moveto(255,390)\curveto(260,400)(370,430)(370,400)
%\curveto(365,360)(270,380)(255,390)
%\closepath}}
%voisinage V
{\pscustom[linewidth=1,linecolor=black,linestyle=dashed,dash=2 4]
{\moveto(270,400)\curveto(270,420)(350,420)(350,400)}}
{\newrgbcolor{curcolor}{0 0 0}
\pscustom[linewidth=1,linecolor=curcolor,fillstyle=solid,fillcolor=white, opacity=0.5]
{\newpath
\moveto(270,400)\curveto(270,450)(350,450)(350,400)
\curveto(350,380)(270,380)(270,400)
\closepath}}
{\pscustom[linewidth=1,linecolor=black,linestyle=dashed,dash=2 4]
{\moveto(350,400)\curveto(350,350)(270,350)(270,400)}}
\rput(298,417){\small $U$}

\psline[linewidth=1](310,400)(310,450)
\psline[linewidth=3](310,450)(310,545)
\psline[linewidth=1]{->}(310,545)(310,580)
\psline[linewidth=1](310,580)(310,590)
\rput(330,505){$L_1$}
\psbezier[linewidth=1](280,410)(320,400)(320,390)(342,390)
\psbezier[linewidth=3](342,390)(395,388)(440,390)(490,430)
\psbezier[linewidth=1]{<<-}(490,430)(500,438)(508,448)(530,480)

\psbezier[linewidth=3](280,410)(240,420)(200,420)(176,433)
\rput(200,553){$K_1$}
\psline[linewidth=0.3](200,540)(220,430)
\psline[linewidth=0.3](205,540)(400,400)
\psline{>>-}(120,460)(176,433)
\psbezier[linewidth=1](120,460)(100,470)(119,460)(120,460)
\end{pspicture}
\caption{Illustration of \cref{t.mainsimplestatementmoregeneral}}

\bigskip
\small 
\centering\vspace*{\fill}
\begin{minipage}{10cm}
Assume that, for all times $0\leq t\leq 1$, the first return linear map $B_t$ admits an $1$-strong unstable, a $1$ and $2$-strong stable manifolds. 
The perturbation $g$ is such that $K_1,K_2$ and $L_1$ are left respectively in the $1$- and $2$-strong stable and $1$-strong unstable manifolds of size $\varrho$ for $g$.
\end{minipage}
\vspace*{\fill}

\end{figure}

\begin{remark}\label{r.localnonlocal}
One could take the compact sets $K_i\subset W^{i,ss}(P,f)\setminus U$ and $L_j\subset W^{j,uu}(P,f)\setminus U$ and replace  
$K_i\subset W^{i,ss}_\varrho(P,g)$ in the conclusions of the theorem by the simpler $$K_i\subset W^{i,ss}(P,g).$$ However this conclusion is strictly weaker, indeed it would give way to possibly annoying situations as depicted in Figure \ref{f.fignonlocal}. 
\end{remark}

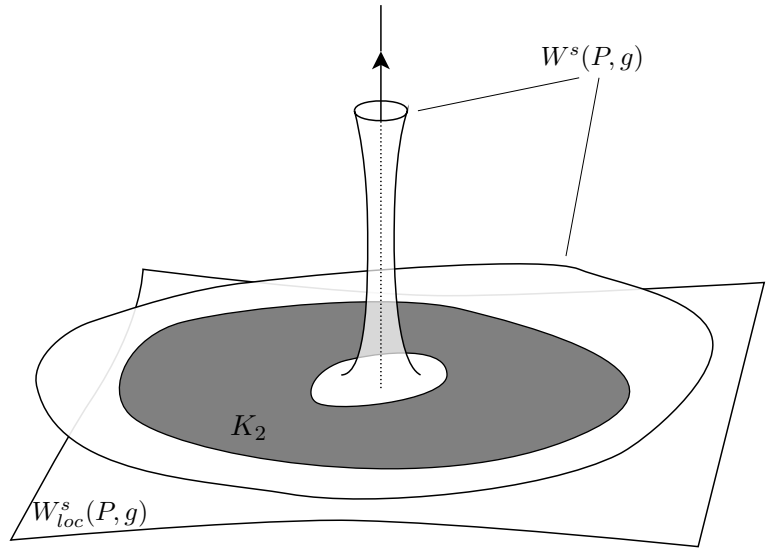
\begin{figure}[hbt] 
\label{f.fignonlocal}
\ifx\JPicScale\undefined\def\JPicScale{1}\fi
\psset{unit=\JPicScale mm}
\psset{linewidth=0.2,dotsep=1,hatchwidth=0.3,hatchsep=1.5,shadowsize=1,dimen=middle}
\psset{dotsize=0.7 2.5,dotscale=1 1,fillcolor=black}
\psset{arrowsize=2 2,arrowlength=1,arrowinset=0.25,tbarsize=0.7 5,bracketlength=0.15,rbracketlength=0.15}
\psset{xunit=.5pt,yunit=.5pt,runit=.5pt}
\begin{pspicture}(0,250)(600,670)

%la vari?t? stable locale
{\pscustom{\newpath \moveto(130,480)
\curveto(130,480)(130,440)(90,380)
\curveto(72,353)(30,275)(30,275)
\curveto(30,275)(196,290)(280,290)
\curveto(370,289)(550,270)(550,270)
\lineto(600,470)
\curveto(600,470)(440,460)(360,460)
\curveto(284,460)(130,480)(130,480)
\closepath}}

%la vari?t? stable retour
{\pscustom[fillstyle=solid,fillcolor=white,opacity=0.9]
{\newpath
\moveto(210,470)
\curveto(168,466)(120,450)(90,440)
\curveto(67,432)(44,410)(50,390)
\curveto(66,323)(173,321)(240,310)
\curveto(302,299)(434,312)(490,340)
\curveto(510,350)(570,400)(560,430)
\curveto(550,460)(485,470)(460,480)
\curveto(435,490)(293,480)(210,470)
\closepath}}
%le compact K
{\pscustom[linewidth=1,linecolor=black,fillstyle=solid,fillcolor=gray,opacity=1]
{\newpath \moveto(370,450)
\curveto(320,462)(208,452)(160,440)\curveto(120,430)(100,390)(120,370)
\curveto(150,340)(337,306)(450,350)\curveto(580,400)(410,440)(370,450)
\closepath }}
\rput(210,360){\large $K_2$}
%le trou dans K
{\pscustom[linewidth=1,linecolor=black,fillstyle=solid,fillcolor=white]
{\newpath
\moveto(280,410)\curveto(258,402)(253,386)(260,380)
\curveto(270,370)(360,380)(360,400)\curveto(360,426)(296,415)(280,410)
\closepath}}
%le tuyau vertical
{\pscustom[fillstyle=solid,fillcolor=white,opacity=0.7]
{\newpath
\moveto(280,400)
\curveto(300,400)(300,470)(300,500)
\curveto(300,520)(300,560)(290,600)
\curveto(290,590)(330,590)(330,600)
\curveto(320,570)(320,520)(320,500)
\curveto(320,490)(320,410)(340,400)
}}

\psbezier(290,600)(290,610)(330,610)(330,600)

\psline[linestyle=dashed,dash=1 2](310,390)(310,592)

\psline{->}(310,592)(310,645)
\psline(310,645)(310,680)
\rput(90,295){$W^s_{loc}(P,g)$}
\rput(470,640){$W^s(P,g)$}
\psline[linewidth=0.3](460,625)(338,600)
\psline[linewidth=0.3](475,625)(450,490)
\end{pspicture}
\caption{The compact $K_2$ may stay in the stable manifold for a perturbation $g$ of $f$ without remaining in the local stable manifold. The stronger conclusions of \cref{t.mainsimplestatementmoregeneral} avoid such picture.}
\end{figure}

Let us give examples of applications of \cref{t.mainsimplestatement,t.mainsimplestatementmoregeneral}. We already knew that the derivative along a saddle of large period may be perturbed in order to get real eigenvalues~\cite{BoCro, BGV}, or that the derivative along a long-period saddle with a weak stable/unstable dominated splitting may be perturbed in order to get a small stable/unstable angle~\cite{BDV}.  In \cref{s.pathcocycles} we show that these perturbations can be obtained through 'good' paths of cocycles, in the sense that one can apply \cref{t.mainsimplestatementmoregeneral} to them. As a consequence, if the period of a saddle is large, then it is possible to perturb it to make the eigenvalues of the first return map real, while preserving their moduli, and preserving the local strong stable and unstable manifolds outside a small neighbourhood:

\begin{theorem}
\label{t.realeigen}
Let $f$ be a diffeomorphism of $M$ and $\epsilon>0$ be a real number. There exists an integer $N\in \NN$ such that for any
\begin{itemize}
\item periodic point $P$ of period $p\geq N$,
\item neighbourhood $U$ of the orbit $\Orb_P$ of $P$, 
\item number $\varrho>0$ and families of compact sets 
\begin{align*}K_i&\subset W^{i,ss}_\varrho(P,f)\setminus \Orb_P, \quad \mbox{ for all }  i\in I\\
 L_j&\subset W^{j,uu}_\varrho(P,f)\setminus \Orb_P,  \quad  \mbox{ for all } j\in J,
 \end{align*} 
 where $I,J$  are the sets of strong stable and unstable dimensions, respectively,
 \end{itemize}
there is a $C^1$-$\epsilon$-perturbation $g$ of $f$ such that
\begin{itemize}
\item  $f^{\pm 1}=g^{\pm 1}$ throughout $\Orb_P$ and outside $U$, 
\item the eigenvalues of the first return map $Dg^p(P)$ are real and their moduli are the same as for $f$,
\item  for all $(i,j)\in I\times J$, we have
 $$K_i\subset W^{i,ss}_\varrho(P,g)\mbox{ and }L_j\subset W^{j,uu}_\varrho(P,g).$$
\end{itemize}
\end{theorem}

Finally, we prove a generic dichotomy between small stable/unstable angles and a weak form of hyperbolicity. Before stating it precisely, we need a few definitions:

A {\em residual} subset of a Baire space is a set that contains a countable intersection of open and dense subsets. 

A {\em saddle point} for a diffeomorphism is a hyperbolic periodic point that has non-trivial stable and unstable manifolds. The {\em index} of a saddle is the dimension of its stable manifold. The {\em stable (resp. unstable) direction} of a saddle $P$ is the tangent vector space to the stable (resp. unstable) manifold at $P$. The {\em minimum stable/unstable angle} of a saddle $P$ is the minimum of the angles between a vector of the stable direction of $P$ and a vector of the unstable direction. 

We say that a saddle point $P$ is {\em homoclinically related} to another saddle point $Q$ if and only if the unstable manifold $W^u(P)$ of the orbit of $P$ (resp. $W^u(Q)$) intersects transversally the stable manifold $W^s(Q)$  (resp. $W^s(P)$) . The {\em homoclinlic class} of a saddle point $P$ is the closure of the transverse intersections of $W^s(P)$ and $W^u(P)$. One easily shows that it also is the closure of the set of saddles homoclinically related to $P$.

A {\em dominated splitting} above a compact invariant set $K$ for a diffeomorphism $f$ is a splitting of the tangent bundle $TM_{|K}=E\oplus F$ into two vector subbundles such that the vectors of $E$ are uniformly exponentially more contracted or less expanded than the vectors of $F$ by the iterates of the dynamics (see definition~\ref{d.dominatedsplitting}). The {\em index} of that dominated splitting is the dimension of $E$. 

For all $1\leq r \leq \infty$, we denote by $\Diff^r(M)$ the space of $C^r$ diffeormorphisms.

\begin{theorem}
\label{t.dichotomysimple}
There exists a residual set $\cR\subset \Diff^1(M)$ of diffeomorphisms $f$ such that for any saddle point $P$ of $f$, we have the following dichotomy:
\begin{itemize}
\item either the homoclinic class  $H(P,f)$ of $P$ admits a dominated splitting of same index as $P$
\item or, for all $\epsilon>0$, there is a saddle point $Q_\epsilon$ homoclinically related to $P$ such that it holds:
\begin{itemize}
\item the minimum stable/unstable angle of
 $Q_\epsilon$ is less than $\epsilon$,
\item the eigenvalues of the derivative of the first return map at $Q_\epsilon$ are all real and pairwise distinct,
\item each of these eigenvalues has modulus less than $\epsilon$ or greater than $\epsilon^{-1}$.
\end{itemize}
\end{itemize}
\end{theorem}

That result parallels~\cite[Theorem 1.1]{Gou}. Indeed, if these three conditions are satisfied for small $\epsilon$, then there are fundamental domains of the stable and unstable manifolds of $Q$ that are big before the distance that separates them, in such a way that these two manifolds can be intertwined by small perturbations. In particular, it is possible to create  tangencies between them by  small perturbations that keep $Q$ in the homoclinic class of $P$. 

We finally give a version of~\cite[Theorem 4.3]{Gou} where the derivative is preserved, that is, we show that if the stable/unstable dominated splitting along a saddle is weak and if the period of that saddle is large, then one obtains homoclinic tangency related to that saddle by a $C^1$-perturbation that preserves the orbit of the saddle and the derivative along it. Moreover, one may keep any preliminarily fixed finite set in the invariant manifolds of the saddle. 

\subsection{Further applications of \cref{t.mainsimplestatement}}
Using Theorem~\ref{t.mainsimplestatement}, Rafael Potrie~\cite{Po} got interesting results on generic Lyapunov stable and bi-stable homoclinic classes. He showed in particular that, $C^1$-generically, if $H$ is a quasi-attractor containing a dissipative periodic point, then it admits a dominated splitting. Also, using a result by Yang~\cite{Y}, he proves that generically  and far from homoclinic tangencies if a homoclinic class is bi-Lyapunov stable (that is, the homoclinic classes that are quasi-attractors and repellors) then it is the whole manifold.

The main theorem of this paper was followed by another result by Bonatti and Bochi~\cite{BoBo} that generalized previous results about perturbation of derivatives along periodic points in $C^1$-topology \cite{M1,BDP,BGV}. More precisely, given a tuple of matrices $\cA=(A_1,...,A_p)$, they give a full description of the tuples of moduli of eigenvalues of the product $B=A_p...A_1$ (equivalently, of Lyapunov exponents) that one can reach by small isotopic perturbations of $\cA$. Moreover, they prove that if strong stable or unstable direction of some dimensions exist at both the initial and final time, then the isotopy $\cA_t$ can be built so that at all times of the isotopy there are strong stable and unstable directions of those dimensions. In other words, the isotopy matches the hypotheses of \cref{t.mainsimplestatement}.

\cref{t.mainsimplestatement} and \cite[Theorem 4.1]{BoBo} thus give a very general method to perturb derivatives inside homoclinic classes, to preserve strong connections and to create new ones. This led recently to a number of developments in the study of $C^1$-generic dynamical systems. Let us detail the most important ones.

In~\cite{BCDG}, Bonatti, Crovisier, D\'iaz and Gourmelon showed a number of generic results on homoclinic classes and produced new examples of wild dynamics. 
In particular, they showed that if a homoclinic class has no dominated splitting and if $C^1$-robustly it contains two saddle points of different indices, then it induces a particular type of wild dynamics, called "viral". Indeed, such a homoclinic class has a replication property: there exists an arbitrarily small $C^1$-perturbation of the dynamics such that there is a new homoclinic class Hausdorff close to the continuation of the first one, but not in the same chain-recurrent class,\footnote{An $\epsilon$-pseudo orbit is a sequence $x_1,...,x_n$ such that $\dist(f(x_i),x_{i+1})<\epsilon$, for all $i$. Two points $x\neq y$ are in the same chain-recurrent class, if for any $\epsilon>0$ there is an $\epsilon$-pseudo orbit that goes from $x$ to $y$ and another that goes from $y$ to $x$. }  and such that that new homoclinic class satisfies the same properties. 

This implies that there is a locally residual set of diffeomorphisms that have uncountably many chain-recurrent classes. By Kupka-Smale's theorem, uncountably many of those chain-recurrent classes have no periodic orbits, that is, are {\em aperiodic}. 

A question since the first production of examples of locally generic dynamics with aperiodic chain-recurrent classes (by Bonatti and D\'iaz~\cite{BD:02}), was whether such aperiodic classes could generically have non-trivial dynamics. It was not known if there could exist locally generic dynamics where aperiodic classes were not all minimal, or had non-zero Lyapunov exponents.

Recently, using (among other ideas) an extension of \cref{t.mainsimplestatement} in dimension $3$, namely the result announced in \cref{s.furtherresults}, using \cite[Theorem 4.1 and Proposition 3.1]{BoBo}, and pushing further the ideas of~\cite{BCDG} and~\cite{BLY}, Bonatti and Shinohara have announced that they can produce open sets of diffeomorphism where  generic diffeomorphisms admit  uncountably many non-minimal chain-recurrent classes.

Moreover, Bonatti, Crovisier and Shinohara~\cite{BCS} proved recently that those techniques can also be used to find a $C^1$-generic counter example to Pesin's theory, thus generalizing the result of Pugh~\cite{Pu84}: for $\dim(M)\geq 3$, they exhibit open sets of $\Diff^1(M)$ in which generic diffeomorphisms admit non-uniformly hyperbolic invariant measures supported by aperiodic chain-recurrent classes that have trivial stable and unstable manifolds.

\subsection{Statement of the Main Perturbation Proposition.}\label{s.mainpertpropos}
We state the main results that lead to \cref{t.mainsimplestatement}. These are perturbation results that hold in $C^r$-topology, for all $1\leq r \leq \infty$, although we only use their $C^1$-versions to prove \cref{t.mainsimplestatement}. The $C^r$ results may be of great interest in other contexts.

We recall that $M$ is a Riemannian manifold, not necessarily compact. We consider the space of diffeomorphisms $\Diff^r(M)$ endowed with the {\em weak Whitney $C^r$-topology}\footnote{See e.g.\ \cite{Hirsch}.}, that is the topology of $C^r$ convergence on compact sets. Notice that if $M$ is compact, this coincides with the usual uniform $C^r$-topology.

While the diffeomorphisms $f\in \Diff^r(M)$ we will consider may vary, they will all coincide throughout the orbit $\Orb_P$ of some common periodic point $P$, and all stable or unstable manifolds of this paper will be those of that orbit. Thus we can unambiguously denote the stable and unstable manifolds of the orbit $\Orb_P$ of $P$ for $f$ simply by $W^s(f)$ and $W^u(f)$. Likewise, we denote
 the $i$-strong stable/unstable manifolds of $\Orb_P$ for $f$ simply by $W^{ss,i}(f)$/$W^{uu,i}(f)$. 
 
Let $P$ be a $p$-periodic point for a diffeomorphism $f$ such that it admits an $i$-strong stable manifold. The same way we defined local stable and unstable manifolds, we may define local strong stable and unstable manifolds:

\begin{definition}
A set $W^{+}(f)$ is a {\em local $i$-strong stable manifold} for $f$ if 
it is a strictly $f$-invariant union of disjoint disks $\{D_{n}\}_{0\leq n<p}$, 
where each $D_{n}$ is a smooth ball inside the strong stable manifold $W^{ss,i}(f)$ and $f^n(P)$ is in the interior of $D_{n}$.
\end{definition}

We define symmetrically a set $W^{-}(f)$ to be a {\em local $j$-strong unstable manifold} for $f$ if it is a local $j$-strong stable manifold for $f^{-1}$. We say that a sequence of compact submanifolds $N_k$ in $M$ (possibly with boundary) {\em $C^r$-converges} to a $C^r$-submanifold $N\subset M$ if there exists a sequence of $C^r$-maps $\phi_k\colon N\to M$ that $C^r$-converges to $\Id_N$ such that $\phi_k(N)=N_k$.  Now, we can do the following:

\begin{remark}
Let $P$ be a periodic point for $f$ and $f_k$ be a sequence in $\Diff^r(M)$ that converges to $f$ for the weak Whitney topology, where each $f_k$ coincides with $f$ throughout $\Orb_P$. Then, by the stable manifold theorem, for any local strong stable manifold $W^+(f)$ there is a sequence of local strong stable manifolds  $W^+(f_k)$ that converges to it $C^r$-uniformly. And symmetrically for local strong unstable manifolds.
\end{remark}

%\bigskip
%
%
%and define a {\em local stable manifold $W^+(f)$ of the orbit of $P$ for $f$} to be an $f$-invariant union of disjoint disks $\{D^s_{n}\}_{0\leq n<p}$, 
%where each $D^s_{n}$ is a smoothly embedded ball inside the stable manifold $W^{s}(f)$ and $f^n(P)$ is in the interior of $D^s_{n}$. 
%
%
%\bigskip
%If P is a hyperbolic periodic point for $f$, if $h_k$ is a sequence that converges to $f$ in $\Diff^r(M)$ and if $h_k$ coincides with $f$ throughout $\Orb_P$, then by the stable manifold theorem, for any local stable manifold $W^+(f)$ there is a sequence of local stable manifolds $W^+(h_k)$ that converges to it $C^r$-uniformly. This is not exactly true anymore if $P$ is not hyperbolic for $f$. Indeed the local stable manifolds for the diffeomorphisms $h_k$ may have strictly greater dimension. For this reason we need  to consider local strong stable manifolds. We follow the notations of~\cite{KH}:

\begin{proposition}[Main perturbation proposition]\label{p.pertpropsimple}
Fix $1\leq r \leq \infty$. Let $g_k$ and $h_k$ be two sequences in $\Diff^r(M)$ converging to a diffeomorphism $f$, such that $f$, $g_k$ and $h_k$ coincide throughout the orbit $\Orb_P$ of a periodic point $P$. Let $\{W^+(h_k)\}_{k\in\NN}$ be a sequence  of local strong stable manifolds of $\Orb_P$ for the diffeomorphisms $h_k$ that converges to a local strong stable manifold $W^+(f)$ for $f$,  $C^r$-uniformly. Define symmetrically local strong unstable manifolds $W^-(h_k)$ and $W^-(f)$.

For any neighborhood $U$ of the orbit $\Orb_P$, there exists:
\begin{itemize}
\item a neighborhood $V\subset U$ of $\Orb_P$,
\item a sequence $f_k$ of $\Diff^r(M)$ converging to $f$,
\item two sequences of local strong stable and unstable manifolds $W^+(f_k)$ and $W^-(f_k)$ of $\Orb_P$ that tend respectively to $W^+(f)$ and $W^-(f)$, in the $C^r$ topology,
\end{itemize}
such that for large $k\in \NN$ it holds:
\begin{itemize}
\item $f_k^{\pm 1}=g_k^{\pm 1}$ inside $V$
\item $f_k^{\pm 1}=h_k^{\pm 1}$ outside $U$,
\item For any integer $i>0$, if $\Orb_P$ has an $i$-strong stable manifold $W^{ss,i}(f)$ for $f$, then $W^{ss,i}(f_k)$ and $W^{ss,i}(h_k)$ also exist and coincide "semilocally outside $U$", i.e.
$$\left[W^+(f_k)\cap W^{ss,i}(f_k)\right]\setminus U=\left[W^+(h_k)\cap W^{ss,i}(h_k)\right]\setminus U,$$
and likewise, replacing stable manifolds by unstable ones. 
\end{itemize}
\end{proposition}

\begin{corollary}[$C^r$-linearization lemma]\label{c.linearisationCr}
Let $1\leq r\leq \infty$. Let $P$ be a periodic hyperbolic point of a diffeormophism  $f\in \Diff^r(M)$ and let $W^+(f)$ and $W^-(f)$ be respectively local stable and unstable manifolds of the orbit $\Orb_P$. Let $U$ be a neighborhood of $\Orb_P$. Then, there exists a sequence $f_k$ tending to $f$ in $\Diff^r(M)$ and two sequences of local strong stable and unstable manifolds $W^+(f_k)$ and $W^-(f_k)$ of $\Orb_P$ such that it holds, for all $k\in \NN$:
\begin{itemize}
\item $f^{\pm 1}=f_k^{\pm 1}$ throughout $\Orb_P$ and outside $U$,
\item $P$ is a hyperbolic point for $f_k$ and the linear part of $f_k^{p}$ at $P$ has no resonances, where $p$ is the period of $P$. In particular, $f_k$ is locally $C^r$-conjugated to its linear part along the orbit of $P$.
\item For any integer $i>0$, if $\Orb_P$ has an $i$-strong stable manifold $W^{ss,i}(f)$ for $f$, then $W^{ss,i}(f_k)$ also exists and
$$\left[W^+(f)\cap W^{ss,i}(f)\right]\setminus U=\left[W^+(f_k)\cap W^{ss,i}(f_k)\right]\setminus U,$$
and likewise, replacing stable manifolds by unstable ones. \end{itemize}
\end{corollary}

In the $C^1$ setting, we have a stronger statement:

\begin{corollary}[$C^1$-linearization lemma]\label{c.linearisationC1}
Let $P$ be a periodic point of a diffeormophism  $f\in \Diff^1(M)$ and let $W^+(f)$ and $W^-(f)$ be respectively local stable and unstable manifolds of $\Orb_P$ and fix a linear structure on a neighborhood of each point of $\Orb_P$. Let $U$ be a neighborhood of $\Orb_P$. Then, there exist a sequence $f_k$ tending to $f$ in $\Diff^1(M)$ and two sequences of local stable and unstable manifolds $W^+(f_k)$ and $W^-(f_k)$ such that it holds, for all $k\in \NN$:
\begin{itemize}
\item $f^{\pm 1}=f_k^{\pm 1}$ throughout $\Orb_P$ and outside $U$, 
\item $f_k$ coincides on a neighborhood of $\Orb_P$ with the linear part $L$ of $f$ along $\Orb_P$,
\item For any integer $i>0$, if $\Orb_P$ has an $i$-strong stable manifold $W^{ss,i}(f)$ for $f$, then $W^{ss,i}(f_k)$ also exists and
$$\left[W^+(f)\cap W^{ss,i}(f)\right]\setminus U=\left[W^+(f_k)\cap W^{ss,i}(f_k)\right]\setminus U,$$
and likewise, replacing stable manifolds by unstable ones. \end{itemize}
\end{corollary}

\cref{p.pertpropsimple} is proved in \cref{s.prop4}. The linearization lemmas are straightforward consequences of \cref{p.pertpropsimple}: use a partition of unity to build a sequence $g_k$ of diffeomorphisms that tends $C^r$ to $f$, such that the linear part of $g_k^{p}$ at $P$ has no resonances (for the proof of \cref{c.linearisationCr}), or such that $g_k^{\pm 1}=L^{\pm 1}$ (for the proof of \cref{c.linearisationC1}) on a neighborhood of the orbit of $P$, where $L$ is the linear part of $f$ along $\Orb_P$, and apply \cref{p.pertpropsimple} with $h_k=f$. In \cref{c.linearisationCr}, the fact that $g_k$ is locally $C^r$-conjugated to its linear part along the orbit of $P$ comes from the Sternberg linearization theorem (see~\cite[Theorem~6.6.6]{KH}).

\subsection{Structure of the paper}
In \cref{s.isotopic}, we prove \cref{t.mainsimplestatement} from \cref{p.pertpropsimple} and \cref{c.linearisationC1}. 
 \cref{p.pertpropsimple} is proved in \cref{s.prop4}. 
 
Finally in \cref{s.consequences} we prove a few of the many consequences of Theorem~\ref{t.mainsimplestatement} for perturbative dynamics of $C^1$ diffeomorphisms. In particular, we prove \cref{t.realeigen,t.dichotomysimple}.

\bigskip

For simplicity, in the rest of the paper, the sentences 
\begin{center}
"For large $k$, property $\cP_k$ holds."

"For small $\lambda>0$, property $\cQ_\lambda$ holds."
\end{center}
\noindent respectively stand for
\begin{center}
 "There exists $k_0\in \NN$ such that, for any integer $k\geq k_0$, property $\cP_k$ holds."

 "There exists $\lambda_0>0$ such that, for any real number $0<\lambda\leq \lambda_0$, property $\cQ_\lambda$ holds." 
\end{center}

\section{Proof of the Isotopic Franks' lemma.}\label{s.isotopic}

In this section, we prove \cref{t.mainsimplestatement} from \cref{p.pertpropsimple} and \cref{c.linearisationC1}.

\begin{proof}[Idea of the proof]
We first put a linear structure on a neighborhood of $\Orb_P$, so that any sequence $\cA_t$ of linear maps as in the statement of \cref{t.mainsimplestatement} identifies to a linear diffeomorphism on the manifold $\ZZ/p\ZZ \times \RR^d$. We will call such diffeomorphisms {\em linear cycle} since they actually do a cyclic permutation on the connected components of $\ZZ/p\ZZ\times \RR^d$.
Then we introduce the notion of "connection" from such a diffeomorphism $\cA$ to another $\cB$, that is, a diffeomorphism of $\ZZ/p\ZZ\times \RR^d$ such that
\begin{itemize}
\item it coincides with $\cB$ on a neighborhood of $\Orb_P$ and with $\cA$ outside a bigger neighborhood, 
\item it "connects" the strong stable/unstable manifolds of $\cA$ with those of $\cB$, as represented in \cref{p.connection}.
\end{itemize}
Those connections can be concatenated as in \cref{p.concatenation} (we may however need to conjugate some of them by homothecies).

As a consequence of \cref{p.pertpropsimple}, if $\cA_t$ is a path of such linear diffeomorphisms along which strong stable and strong unstable manifolds of dimensions $i\in I$ and $j\in J$ exist,  then we will find a sequence $0=t_0<t_1<...<t_k=1$ of times such that there is a connection from each $\cA_{t_i}$ to $\cA_{t_{i+1}}$ of small size (that is, $C^1$-close to the linear cycle $\cA_{t_i}$) that connects the $I$-strong stable and $J$-strong unstable manifolds of $\cA_{t_i}$ to those $\cA_{t_{i+1}}$. Then a convenient concatenation of those connections (see \cref{p.concatenation}) will give a connection from $\cA_0$ to $\cA_1$ whose distance to $\cA_0$ will be arbitrarily close to the radius of the path $\cA_t$, as defined in  \cref{t.mainsimplestatement}.

We will end the proof by linearizing $f$ to $\cA_0$ on a neighborhood $U$ of $\Orb_P$ with \cref{c.linearisationC1}, and finally pasting in $U$ that connection from $\cA_0$ to $\cA_1$. 
\end{proof}

We now go into the details of the proof, starting with a precise definition of the metrics we deal with.

\subsection{Preliminaries}
Any Riemannian metric $\|.\|_*$ on a manifold $M$ induces a distance $d_{\|.\|_*}$ on $TM$ through the Levi-Civita connection. 
We define the corresponding $C^1$-distance between two diffeomorphisms $g,h\colon M \to M$ as follows:
$$\dist_{\|.\|_*}(g,h)=\sup_{v\in TM\atop \|v\|_*=1}\biggl\{d_{\|.\|_*}\bigl[Dg(v),Dh(v)\bigr],d_{\|.\|_*}\bigl[Dg^{-1}(v),Dh^{-1}(v)\bigr]\biggr\}.$$
We say that a diffeomorphism $g$ of $M$ is {\em bounded by $C>1$ for $\|.\|_*$} if for all unit vector $v\in TM$, we have $C^{-1}\leq \|Df(v)\|_*\leq C$.

We will deal with diffeomorphisms of the following $d$-dimensional manifold (and vector bundle):
$$\cE=\ZZ/p\ZZ\times \RR^d.$$
 A {\em cyclic diffeomorphism $g$} of $\cE$ is a diffeomorphism of $\cE$ such that for each integer $1\leq n<p$, 
\begin{align*}
g(\{n\}\times \RR^d)&=\{n+1\}\times \RR^d\\
g(0_\cE)&=0_\cE,
\end{align*}
where $0_\cE$ is the $0$-section of the vector bundle $\cE$.

Whenever they exist, we denote the $i$-strong stable and unstable manifolds of the periodic orbit $0_\cE$ of $g$ by $W^{ss,i}(g)$ and $W^{uu,i}(g)$, respectively.

A cyclic diffeomorphism of $\cE$  that restricts to a linear map from $\{n\}\times \RR^d$ to $\{n+1\}\times \RR^d$, for all $n$, is called a {\em linear cycle}. Denote by $\mathfrak{A}$ the set of linear cycles of $\cE$. Notice that $\mathfrak{A}$ identifies to the set of tuples $(A_1,...,A_p)\in GL(d,\RR)^p$.

We endow the manifold $\cE$ with the Riemannian metric $\|.\|_{\mbox{Eucl.}}$ arising from the canonical Euclidean structure on $\RR^d$.

The $C^1$-distance $\dist_{\|.\|_{\mbox{Eucl.}}}(\cA,\cB)$ between two distinct linear cycles $\cA,\cB\in \mathfrak{A}$, seen as diffeomorphisms of $\cE$, is infinite.
We endow $\mathfrak{A}$ with the more appropriate operator distance: given $\cA=(A_1,...,A_p)$ and $\cB=(B_1,...,B_p)$ in $\mathfrak{A}$, let 
\begin{align*}
\dist_{\mathfrak{A}}(\cA,\cB)=\max_{1\leq n\leq p}\bigl\{\|A_n-B_n\|, \|A_n^{-1}-B_n^{-1}\|\bigr\},
\end{align*}
where $\|.\|$ is the operator norm. For any finite sets $I$ and $J$ of strictly positive integers, denote by
\begin{align*}
\mathfrak{A}_{I,J}\subset \mathfrak{A}
\end{align*} the subset of tuples $(A_{1},\ldots,A_{p})\in \mathfrak{A}$ such that the endomorphism $B=A_{p}\circ ... \circ A_{1}$ has an $i$-strong stable direction and a $j$-strong unstable direction, for all $i\in I$ and $j\in J$.

\subsection{Connections from a linear cycle to another.}

\begin{definition}
Given two linear cycles $\cA,\cB\in \mathfrak{A}_{I,J}$, an {\em $(I,J)$-connection $C_{\cA\cB}$ from $\cA$ to $\cB$} is a cyclic diffeomorphism on $\cE$ such that 
\begin{itemize}
\item $C_{\cA\cB}^{\pm 1}=\cA^{\pm 1}\mbox{ outside a bounded neighborhood of $0_\cE$},$ 
\item $C_{\cA\cB}^{\pm 1}=\cB^{\pm 1}\mbox{ on a neighborhood of $0_\cE$},$
\item there exists a neighborhood $U$ of $0_\cE$ such that the strong stable/unstable manifolds of the periodic orbit $\Orb_P=0_\cE$ for $C_{\cA\cB}$ and $\cA$ coincide outside $U$, that is, for all $i\in I$ and $j\in J$ it holds:
\begin{align*}
W^{ss,i}(\cA)\setminus U&=W^{ss,i}(C_{\cA\cB})\setminus U,\\
W^{uu,j}(\cA)\setminus U&=W^{uu,j}(C_{\cA\cB})\setminus U.
\end{align*}
\end{itemize}
\end{definition}

\begin{remark}\label{r.proper immersion}
The manifolds $W^{ss,i}(C_{\cA\cB})$ and $W^{uu,j}(C_{\cA\cB})$ are then properly immersed in $\cE$. In particular, the negative iterates of any point of  $W^{ss,i}(C_{\cA\cB})$ and the positive iterates of any point of  $W^{uu,i}(C_{\cA\cB})$ tend to infinity.
\end{remark}

\begin{figure}[hbt] 
\ifx\JPicScale\undefined\def\JPicScale{1}\fi
\psset{linewidth=0.011,dotsep=1,hatchwidth=0.3,hatchsep=1.5,shadowsize=0,dimen=middle}
\psset{dotsize=0.1 2.5,dotscale=1 1,fillcolor=black}
\psset{arrowsize=0.17 2,arrowlength=1,arrowinset=0.25,tbarsize=0.7 5,bracketlength=0.15,rbracketlength=0.15}
\psset{xunit=.375pt,yunit=.290pt,runit=.325pt}
\begin{pspicture}(0,450)(700,800)

%contour ext?rieur
{\pscustom[linewidth=1,linecolor=black,fillstyle=solid,fillcolor=lightgray,opacity=1]
{\newpath\moveto(404.66271961,807.13292936)
\curveto(376.04996481,840.88627171)(325.16950605,836.63862931)(289.17248185,818.26997942)
\curveto(217.86596661,782.54966282)(160.71502647,713.86762563)(145.8536995,634.71344846)
\curveto(135.25868902,580.92641969)(160.73681693,524.22856443)(205.91109676,493.9223568)
\curveto(285.56364066,439.15668389)(392.46477366,424.86964647)(483.64315824,457.13355891)
\curveto(525.74046706,471.79976812)(570.10664183,507.75528318)(565.81165852,556.85241065)
\curveto(561.8447513,591.57079837)(537.63211722,619.35382667)(521.02303586,648.88180941)
\curveto(485.09950761,703.51556697)(448.72277334,758.62766863)(404.66271961,807.13292936)
\closepath}}

%contour interieur
{\pscustom[linewidth=1,linecolor=black,fillstyle=solid,fillcolor=white,opacity=1]
{\newpath\moveto(385.69799961,708.94769936)
\curveto(350.08683576,718.01521442)(306.71552825,717.65869953)(282.48331784,677.17166076)
\curveto(253.3261646,644.22750459)(280.40229059,607.12037405)(313.94642459,584.65770724)
\curveto(339.07704789,562.41255047)(374.78603848,537.53578227)(409.68788295,550.7106263)
\curveto(448.56762661,569.10586097)(446.89433626,610.02186718)(435.17320472,640.27167335)
\curveto(435.59919878,668.02648353)(419.39800568,707.39344747)(385.69799961,708.94769936)
\closepath}}

%trait central SO-NE et branche est
\psline{<->}(324.77707001,595.11885216)(407.56768061,663.98419136)
\psline(312.77707001,584.51885216)(407.56768061,663.98419136)

{\pscustom[linewidth=1,linecolor=black]
{\newpath\moveto(407.56768061,663.98419136)
\curveto(427.00134509,681.12144503)(461.83593158,698.66364379)(481.45143231,672.11999076)
\moveto(481.45143231,672.11999076)
\curveto(496.78602184,640.1838167)(532.60707574,646.18229266)(560.66593864,654.09654128)
\curveto(603.63847672,662.16487749)(646.56764388,670.4625773)(689.49121772,678.78681364)}}

%trait central NO-SE
\psline(287.02311121,683.19374056)(419.83258701,558.74839866)
\psline{>-<}(297.02311121,673.19374056)(409.83258701,568.74839866)

%{\pscustom[linewidth=1,linecolor=black]
%{\newpath\moveto(287.02311121,683.19374056)
%\lineto(419.83258701,558.74839866)\closepath}}

%branche SE
{\pscustom[linewidth=1,linecolor=black]
{\newpath\moveto(419.83258701,558.74839866)
\curveto(441.56690219,533.04161974)(476.88006248,531.40389279)(507.18488204,539.66208378)
\curveto(530.45132887,543.05662031)(551.38628039,534.60332551)(571.54685491,524.08917333)
\lineto(705.02833161,463.94129336)}}

%branche NO
{\pscustom[linewidth=1,linecolor=black]
{\newpath\moveto(287.02311121,683.19374056)
\curveto(267.67505917,703.90996833)(235.75639908,723.85671236)(208.5216749,704.59903501)
\curveto(181.79371329,683.70869447)(154.28138029,713.53196922)(128.29230187,721.55623514)
\curveto(83.14172661,742.55643871)(38.04292791,763.66782103)(-7.04292039,784.80657936)}}

%branche SO
{\pscustom[linewidth=1,linecolor=black]
{\newpath\moveto(312.77707001,584.51885216)
\curveto(291.92105714,563.61296892)(257.39797229,547.26862913)(233.09245261,573.12721065)
\curveto(207.03787083,588.69137898)(174.45293053,588.38937296)(145.08202451,586.35507116)
\lineto(-9.49121239,561.21318536)}}

\rput(350,500){$U$}
\rput(620,790){\small $C_{\cA\cB}=\cA$}
\rput(360,598){\Small $0_\cE$}
\rput(350,682){\SMALL $C_{\cA\cB}\!=\!\cB$}
\rput(10,530){\Small $W^{u}(\cA)\!\setminus \!U=W^{u}(C_{\cA\cB})\!\setminus \!U$}
\rput(10,800){\SMALL $W^{s}(\cA)\!\setminus \!U=W^{s}(C_{\cA\cB})\!\setminus \!U$}

\end{pspicture}
\caption{A connection from a linear cycle $\cA$ to another $\cB$.}
\label{p.connection}
\end{figure}

Notice that, since a connection $C_{\cA\cB}$ from a linear cycle $\cA$ to another $\cB$ coincides with $\cA$ outside a compact set, the $C^1$-distance between $\cA$ and $C_{\cA\cB}$ is bounded. We define the {\em size} of the connection $C_{\cA\cB}$ as the $C^1$-distance for the Euclidean metric:
\begin{align*}
\size(C_{\cA\cB})=\dist_{\|.\|_{\!\mbox{\scriptsize Eucl.}}}(\cA,C_{\cA\cB}).
\end{align*}

For $\lambda>0$, denote by $\lambda.\Id$ the homothety of ratio $\lambda$ on $\cE$, that is, the diffeomorphism of $\cE$ defined by $\lambda.\Id(n,x)=(n,\lambda.x)$. Given a connection $C_{\cA\cB}$ from $\cA$ to $\cB$, we denote by $C^\lambda_{\cA\cB}$ its conjugate by $\lambda.\Id$:
$$C^\lambda_{\cA\cB}=\lambda.\Id\circ C_{\cA\cB}\circ \lambda^{-1}.\Id.$$

\begin{remark}\label{r.homothetic}
The strong stable/unstable manifolds of $C^\lambda_{\cA\cB}$ are the homothetic images of those of $C_{\cA\cB}$: for all $i\in I, j\in J$, it holds:
\begin{align*}
W^{ss,i}(C^\lambda_{\cA\cB})&=\lambda.W^{ss,i}(C_{\cA\cB}),\\
W^{uu,i}(C^\lambda_{\cA\cB})&=\lambda.W^{uu,i}(C_{\cA\cB}).
\end{align*}
\end{remark}

We state without a proof the following easy result:

\begin{lemma}
If $C_{\cA\cB}$ is an $(I,J)$-connection from $\cA$ to $\cB$, then for any $\lambda>0$,  $C^\lambda_{\cA\cB}$ is also one. Moreover, for all $0<\lambda<\mu$, we have:
\begin{align*}
\size(C^\lambda_{\cA\cB})<\size(C^\mu_{\cA\cB}).
\end{align*}
\end{lemma}

\begin{lemma}\label{l.concat}
Let $\cA,\cB,\cC\in \mathfrak{A}_{I,J}$. Given $(I,J)$-connections $C_{\cA\cB}$ and $C_{\cB\cC}$ from $\cA$ to $\cB$ and from $\cB$ to $\cC$, respectively, 
for all $\epsilon>0$, there exists an $(I,J)$-connection $C_{\cA\cC}$ from $\cA$ to $\cC$ such that one has the following inequalities:
\begin{align}
\size(C_{\cA\cC})&\leq \max\bigl\{\size(C_{\cA\cB}),\size(C_{\cB\cC})+ \dist_{\mathfrak{A}}(\cA,\cB)+\epsilon\bigr\},\label{e.zvvve}
\end{align}
\end{lemma}

The idea of the proof is very natural: just paste a conjugate $C^\lambda_{\cB\cC}$ of the connection from $\cB$ to $\cC$ into the region where $C_{\cA\cB}=\cB$. The cyclic diffeomorphism thus obtained is however not necessarily a connection from $\cA$ to $\cC$. One needs $\lambda>0$ to be small enough for the stable and unstable manifolds of the new cyclic diffeomorphism to behave as expected. Some technical work is needed here, as the proof of \cref{c.fz} shows.  

\begin{proof}
Let $U$ be a bounded neighborhood of $0_\cE$ such that 
\begin{align}
C_{\cA\cB}^{\pm 1}=\cB^{\pm 1}\mbox{ on $U$.}\label{e.rqg}
\end{align}
Define the map $C_{\cA\cC,\lambda}\colon \cE\to \cE$ by 
\begin{align}
C_{\cA\cC,\lambda}&=C_{\cB\cC}^{\lambda} \mbox{ on $U$}\label{e.zfiofzp}\\
&=C_{\cA\cB} \mbox{ outside $U$.}\label{e.zfiofzp2}
\end{align}

\begin{claim}\label{c.fz}
Let $i\in I$. For small $\lambda>0$, the map $C_{\cA\cC,\lambda}$ is an $(\{i\},\emptyset)$-connection from $\cA$ to $\cC$.
\end{claim}

\begin{proof}
Let $V$ be the maximal bounded neighborhood $V$ of $0_\cE$ on which $C_{\cB\cC}\neq\cB$. Then $C^\lambda_{\cB\cC}=\cB^\lambda=\cB$ outside $\lambda V=\lambda.\Id(V)$. As a consequence, for small $\lambda>0$ (for $\lambda V\subset U$), the map  $C_{\cA\cC,\lambda}$ is a cyclic diffeomorphism on $\cE$.
\bigskip

We are left to show that for small $\lambda>0$, the $i$-strong stable manifolds of $C_{\cA\cC,\lambda}$ and $\cA$ coincide outside a bounded neighborhood of $0_\cE$.
Let $\cD$ be a bounded fundamental domain of $W^{ss,i}(\cB)$ (we ask that its closure $\overline{\cD}$ be a compact set of $W^{ss,i}(C_{\cA\cB})\setminus \{0_\cE\}$) such that 
\begin{align}
\cB^n(\cD)\subset U \mbox{   for all $n\geq 0$}.\label{e.fpojfzef} 
\end{align}
Let us prove that, for small $\lambda>0$, we have:
\begin{itemize}
\item[$(i)$] $\cD$ is a fundamental domain of $W^{ss,i}(\cC^\lambda_{\cB\cC})$,\label{i.1}
\item[$(ii)$] $(\cC^\lambda_{\cB\cC})^{-n}=\cB^{-n}$ by restriction to $\cD$, for all $n\geq 0$.\label{i.2}
\item[$(iii)$] $(\cC_{\cB\cC}^\lambda)^n(\cD)\subset U$, for all $n\geq 0$.\label{i.3}
\item[$(iv)$] $\cD$ is a fundamental domain of $W^{ss,i}(C_{\cA\cC,\lambda})$,\label{i.4}
\end{itemize}
\begin{itemize}
\item[Proof of $(i)$:]
It comes from \cref{r.homothetic}, from $\cD$ being a fundamental domain of $W^{ss,i}(\cB)$ and from  $C_{\cB\cC}$ and $W^{ss,i}(C_{\cB\cC})$ coinciding respectively with $\cB$ and $W^{ss,i}(\cB)$ outside a bounded neighborhood of $0_\cE$. 

\item[Proof of $(ii)$:]
The union of negative $\cB$-iterates of the stable fundamental domain $\cD$ is bounded away from $0_\cE$, hence for small $\lambda>0$, $(\cC^\lambda_{\cB\cC})^{-1}=\cB^{-1}$ on that union, and $(ii)$ holds.

\item[Proof of $(iii)$:]
By $(i)$, the set
$$\Si=\bigcup_{n\geq 0}{(C^\lambda_{\cB\cC})}^n(\cD)$$
is bounded. Hence, for small $0<\nu<1$,
\begin{align}
\nu.\Si&=\bigcup_{n\geq 0}{(C^{\lambda.\nu}_{\cB\cC})}^n(\nu.\cD)\nonumber\\
&\subset U.\label{e.iunezd2}
\end{align}
Conjugating the equality in $(ii)$ by $\nu.\Id$, we get
\begin{align}
(C^{\lambda.\nu}_{\cB\cC})^{-n}&=\cB^{-n}\mbox{ by restriction to $\nu.\cD$, for all $n\geq 0$.}\label{e.qege}
\end{align}
For all $x\in\cD$, there is a (unique) positive $\cB$-iterate of $x$ in $\nu.\cD$. By \cref{e.qege}, the  positive $(C^{\lambda.\nu}_{\cB\cC})$-iterates of $x$ coincide with the $\cB$-iterates of $x$ until that iterate in $\nu.\cD$. \cref{e.fpojfzef} gives that those iterates are in $U$. By \cref{e.iunezd2}, the next iterates (after reaching $\nu.\cD$) also lie in $U$. Therefore, $(iii)$ holds, for small $\lambda:=\lambda.\nu$, hence for small $\lambda$.

\item[Proof of $(iv)$:]
It is a straightforward consequence of \cref{e.zfiofzp} and Items $(i)$ and $(iii)$.
\end{itemize}

\medskip

By \cref{r.proper immersion}, one finds a bounded neighborhood $\hat{U}\subset U$ of $0_\cE$ such that
\begin{align}
C_{\cA\cB}^{-n}(\cD)\cap \hat{U}=\emptyset, \mbox{ for all $n\geq 0$.}\label{e.qerge}
\end{align}

For small $\lambda>0$, it holds
\begin{itemize}
\item[$(v)$] $C_{\cB\cC}^\lambda=\cB$ outside $\hat{U}$, 
\item[$(vi)$] $C_{\cA\cC,\lambda}=\cC_{\cA\cB}$ outside $\hat{U}$,
\item[$(vii)$] $C_{\cA\cC,\lambda}^{-n}(\cD)=C_{\cA\cB}^{-n}(\cD)$, for all $n\geq 0$.
\end{itemize}

Item $(v)$ is trivial. We deduce $(vi)$ from $(v)$ and \cref{e.rqg,e.zfiofzp,e.zfiofzp2}. Last, with \cref{e.qerge} we get $(vii)$.
\medskip

Items $(iv)$ and $(vii)$ give that, for small $\lambda>0$, the $i$-strong stable manifold of $C_{\cA\cC,\lambda}$ coincides with that of $C_{\cA\cB}$, hence with that of $\cA$, outside a bounded neighborhood of $0_\cE$. This ends the proof of the claim.
\end{proof}

The symmetrical claim, under time inversion also holds:

\begin{claim}
Let $j\in J$. For small $\lambda>0$, the map $C_{\cA\cC,\lambda}$ is an $(\emptyset,\{j\})$-connection from $\cA$ to $\cC$.
\end{claim}

As a consequence of both claims, for small $\lambda>0$, the map $C_{\cA\cC,\lambda}$ is an $(I,J)$-connection from $\cA$ to $\cC$.
\bigskip

We are left to compute the size of the connection $C_{\cA\cC,\lambda}$.
We recall $V\subset \cE$ is the bounded set on which $C_{\cB\cC}\neq\cB$. Notice that $C_{\cA\cC,\lambda}=C_{\cA\cB}$ outside $\lambda V$.
Let $v\in T\cE$ be a unit vector above $M\setminus \lambda V$. 
Then 
\begin{align*}d_{\|.\|_{\!\mbox{\scriptsize Eucl.}}}(DC_{\cA\cC,\lambda}(v),D\cA(v))&=d_{\|.\|_{\!\mbox{\scriptsize Eucl.}}}(DC_{\cA\cB}(v),D\cA(v))\\
&\leq \size(C_{\cA\cB}).
\end{align*}
Let $v\in T\cE$ be a unit vector above $\lambda V$. Then
\begin{align*}
d_{\|.\|_{\!\mbox{\scriptsize Eucl.}}}(DC_{\cA\cC,\lambda}(v),D\cA(v))&=d_{\|.\|_{\!\mbox{\scriptsize Eucl.}}}(DC^\lambda_{\cB\cC}(v),D\cA(v))\\
&\leq d_{\|.\|_{\!\mbox{\scriptsize Eucl.}}}(DC^\lambda_{\cB\cC}(v),D\cB(v))+d_{\|.\|_{\!\mbox{\scriptsize Eucl.}}}(D\cB(v),D\cA(v))\\
&\leq \size(C^\lambda_{\cB\cC})+d_{\|.\|_{\!\mbox{\scriptsize Eucl.}}}(D\cB(v),D\cA(v))\\
&\leq \size(C_{\cB\cC})+\dist_{\mathfrak{A}}(\cA,\cB)+d_\lambda,
\end{align*} 
where $d_\lambda$ is a quantity that only depends on the size of the bounded set $\lambda V$ and that tends to zero when $\lambda$ goes to zero. 
%Thus when $\lambda>0$ is less than some $\lambda_0$, one has $d_{\|.\|_{\!\mbox{\scriptsize Eucl.}}}(Dh(v),D\cA(v))\leq 

One may do the same study, looking at the preimages of the unit vectors $w\in T\cE$. We finally get that for small $\lambda>0$, \cref{e.zvvve} holds. This ends the proof of \cref{l.concat}.
\end{proof}

\begin{corollarysection}\label{c.concat?nation}
Let $\cA_1,...,\cA_\ell$ be a sequence of linear cycles in $\mathfrak{A}_{I,J}$ and, for all $1\leq n<\ell$, let $C_{\cA_n\cA_{n+1}}$ be an $(I,J)$-connection 
  from $\cA_n$ to $\cA_{n+1}$. 

Then, for all $\epsilon>0$, there exists an $(I,J)$-connection $C_{\cA_1\cA_{\ell}}$ from $\cA_1$ to $\cA_\ell$  such that 
\begin{align*}
\size(C_{\cA_1\cA_{\ell}})\leq \max_{1\leq n<\ell}\left\{\size(C_{\cA_n\cA_{n+1}})+ \dist_{\mathfrak{A}}(\cA_1,\cA_n)+\epsilon\right\}.
\end{align*}
\end{corollarysection}

\begin{proof}
The proof is a straightforward induction on $\ell$ applying \cref{l.concat} at each step. The connection $C_{\cA_1\cA_{\ell}}$ is depicted in \cref{p.concatenation}.
\end{proof}

\begin{figure}[hbt] 
\ifx\JPicScale\undefined\def\JPicScale{1}\fi
\psset{linewidth=0.011,dotsep=1,hatchwidth=0.3,hatchsep=1.5,shadowsize=0,dimen=middle}
\psset{dotsize=0.1 2.5,dotscale=1 1,fillcolor=black}
\psset{arrowsize=0.17 2,arrowlength=1,arrowinset=0.25,tbarsize=0.7 5,bracketlength=0.15,rbracketlength=0.15}
\psset{xunit=.425pt,yunit=.425pt,runit=.425pt}
\begin{pspicture}(200,550)(500,950)
%{\pscustom[linewidth=1,linecolor=black]
%{\newpath
%\moveto(260,560.00000262)
%\curveto(100,580.00000262)(30.375359,751.69501262)(100,860.00000262)
%\curveto(306.30348,1180.91652262)(867.30026,558.89960262)(260,560.00000262)
%\closepath}}

{\pscustom[linewidth=1,linecolor=black,fillstyle=solid,fillcolor=lightgray]
{\newpath
\moveto(150,840.00000262)
\curveto(129.69783,806.84984262)(86.4566,778.71118262)(90,740.00000262)
\curveto(96.888622,664.74272262)(166.05239,585.58401262)(240,570.00000262)
\curveto(286.70029,560.15820262)(326.98434,609.32560262)(370,630.00000262)
\curveto(417.02542,652.60158262)(490,640.00000262)(510,700.00000262)
\curveto(525.20234,745.60701262)(430,751.92598262)(430,800.00000262)
\curveto(430,900.00000262)(374.51472,888.22857262)(330,910.00000262)
\curveto(290.95801,929.09480262)(238.79368,939.59436262)(200,920.00000262)
\curveto(171.9307,905.82243262)(166.42358,866.81706262)(150,840.00000262)
\closepath}}

{\pscustom[linewidth=1,linecolor=black,fillstyle=solid,fillcolor=white]
{\newpath
\moveto(430,700.00000262)
\curveto(413.41409,654.87773262)(355.33272,636.00175262)(310,620.00000262)
\curveto(275.28158,607.74493262)(227.44009,603.13997262)(200,610.00000262)
\curveto(149.74877,622.56280262)(111.44596,734.29730262)(150,760.00000262)
\curveto(180,780.00000262)(158.77659,824.50231262)(170,850.00000262)
\curveto(181.22341,875.49768262)(204.35539,903.98500262)(270,900.00000262)
\curveto(335.64461,896.01499262)(368.35148,874.63764262)(390,840.00000262)
\curveto(419.55466,792.71254262)(446.74463,745.55406262)(430,700.00000262)
\closepath}}

{\pscustom[linewidth=1,linecolor=black,fillstyle=solid,fillcolor=lightgray]
{\newpath
\moveto(220,850.00000262)
\curveto(167.57211,801.65443262)(185.76194,604.60955262)(230,640.00000262)
\curveto(280,680.00000262)(360.62308,641.24615262)(390,700.00000262)
\curveto(410,740.00000262)(370.60413,829.92333262)(340,850.00000262)
\curveto(309.39587,870.07666262)(246.51475,874.45016262)(220,850.00000262)
\closepath}}

{\pscustom[linewidth=1,linecolor=black,fillstyle=solid,fillcolor=white]
{\newpath
\moveto(240,830.00000262)
\curveto(261.82562,834.12821262)(310,840.00000262)(340,810.00000262)
\curveto(370,780.00000262)(349.8124,724.71859262)(340,710.00000262)
\curveto(321.04605,681.56908262)(206.07112,677.39343262)(206.07112,715.18800262)
\curveto(206.07112,734.15235262)(209.11379,739.11379262)(220,750.00000262)
\curveto(230.88621,760.88620262)(217.00658,769.60720262)(214.15234,783.87837262)
\curveto(211.2981,798.14954262)(218.17438,825.87179262)(240,830.00000262)
\closepath}}

{\pscustom[linewidth=1,linecolor=black,fillstyle=solid,fillcolor=white]
{\newpath
\moveto(270,770.00000262)
\curveto(193.77922,748.02821262)(309.10668,660.28003262)(309.10668,765.69563262)
\curveto(309.10668,778.19052262)(285.19928,775.74531262)(270,770.00000262)
\closepath}}

\psline{-*}(270,743)(270,743)

\rput(130,860){$\cA_1$}
\rput(190,845){$\cA_2$}
\rput(236,810){$\cA_3$}
\rput(255,791){$\ddots$}
\rput(292,763){\small $\cA_\ell$}
\rput(280,730){\SMALL $0_\cE$}
\rput(240,675){$C_{\cA_2\cA_3}$}
\rput(255,590){$C_{\cA_1\cA_2}$}

\end{pspicture}
\caption{A connection from a linear cycle $\cA_1$ to another $\cA_\ell$ obtained from a sequence of connections $(C_{\cA_n\cA_{n+1}})_{1\leq n<\ell}$.}
\label{p.concatenation}
\end{figure}
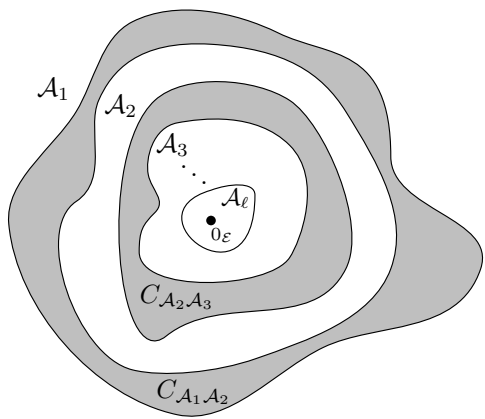

For all $\cA,\cB \in  \mathfrak{A}_{I,J}$, let $d_{I,J}(\cA\to\cB)$ be the infimum of the sizes of the $(I,J)$-connections from $\cA$ to $\cB$, and $+\infty$ if there is no such connection. Let $d_{I,J}(\cA,\cB)=\max\bigl\{d_{I,J}(\cA\to\cB),d_{I,J}(\cB\to\cA)\bigr\}$. This defines a distance on $\mathfrak{A}_{I,J}$: the triangle inequality comes from \cref{l.concat} and the fact that 
$$\dist_{\mathfrak{A}}(\cA,\cB)\leq \size(C_{\cA\cB}),$$
moreover  $d_{I,J}(\cA,\cA)=0$ since  $\cA$ is a trivial connection from $\cA$ to $\cA$.

\begin{proposition}\label{p.equivdistance}
The topologies induced  on $ \mathfrak{A}_{I,J}$ by the distances $d_{I,J}$ and $\dist_\mathfrak{A}$ coincide. 
\end{proposition}

\begin{proof}
The fact that $d_{I,J}\geq \dist_\mathfrak{A}$ is clear. We are left to show that, if a sequence $\cA_k\in  \mathfrak{A}_{I,J}$ converges to $\cA\in  \mathfrak{A}_{I,J}$ for the distance $\dist_\mathfrak{A}$, then it also does for $d_{I,J}$, that is, $d_{I,J}(\cA_k\to \cA)$ and $d_{I,J}(\cA\to \cA_k)$ both tend to $0$.

Fix two sequences $\cA_k,\cB_k\in  \mathfrak{A}_{I,J}$ converging to $\cA\in  \mathfrak{A}_{I,J}$, it is enough to show that the sequence $d_{I,J}(\cA_k\to \cB_k)$ converges to zero. Notice that the corresponding cyclic diffeomorphisms $\cA_k$ and $\cB_k$ converge to $\cA$ for the weak Whitney $C^1$-topology. 

Let $W^+(\cA)$ and $W^-(\cA)$ be local stable and unstable manifolds for $\cA$. For all $k\in \NN$, let $W^+(\cA_k)$ and $W^-(\cA_k)$ be local (strong)\footnote{The stable dimensions of $\cA_k$ and $\cA$ may differ for all $k$, however we know that for large $k$, the stable dimension of $\cA$ is a strong stable dimension of $\cA_k$.} stable and unstable manifolds for $\cA_k$ so that the sequences $W^+(\cA_k)$ and $W^-(\cA_k)$ converge $C^1$-uniformly to $W^{+}(\cA)$ and $W^{-}(\cA)$, respectively. 

\medskip

The rest of the proof is natural: applying \cref{p.pertpropsimple} to the sequences of diffeomorphisms $\cA_k$ and $\cB_k$ with the sequences  $W^+(\cA_k)$ and $W^-(\cA_k)$ and with a neighborhood $U$ of the orbit $0_\cE$ small enough, we obtain a sequence of diffeomorphisms $h_k=C_{\cA_k\cB_k}$ such that for large $k$, $C_{\cA_k\cB_k}$ is an $(I,J)$-connection from $\cA_k$ to $\cB_k$ whose size tends to $0$ as $k\to \infty$.

\medskip
We go into the details. Take a neighborhood $U$ of $0_\cE$ such that there is
\begin{itemize}
\item[(i)] a fundamental domain $\cD^+$ of $W^+(\cA)$ whose closure lies in the interior of $W^+(\cA)$ and such that the closed set $\overline{\cup_{n\in \NN}\cA^{-n}(\cD^+)}$ does not intersect the closure of $U$,
\item[(ii)] a fundamental domain $\cD^-$ of $W^-(\cA)$ whose closure lies in the interior of $W^-(\cA)$ and such that the closed set $\overline{\cup_{n\in \NN}\cA^{n}(\cD^-)}$ does not intersect the closure of $U$.  
\end{itemize}
\cref{p.pertpropsimple} gives 
\begin{itemize}
\item[1.] a sequence of diffeomorphisms $C_{\cA_k\cB_k}\in \Diff^1(M)$ converging to $\cA$ and coinciding with $\cA$ throughout $O_\cE$,
\item[2.] two sequences of local strong stable and unstable manifolds $W^+(C_{\cA_k\cB_k})$ and $W^-(C_{\cA_k\cB_k})$ for $C_{\cA_k\cB_k}$,
\end{itemize}
 such that for large $k\in \NN$ we have:
\begin{itemize}
\item[3.] $C_{\cA_k\cB_k}^{\pm 1}=\cB_k^{\pm 1}$ on some neighborhood $V$ of $O_\cE$,
\item[4.] $C_{\cA_k\cB_k}^{\pm 1}=\cA_k^{\pm 1}$ outside $U$,
\item[5.] For any integer $i\in I$, we have $$\left[W^+(C_{\cA_k\cB_k})\cap W^{ss,i}(C_{\cA_k\cB_k})\right]\setminus U=\left[W^+(\cA_k)\cap W^{ss,i}(\cA_k)\right]\setminus U,$$
and likewise, replacing $I$ by $J$ and stable manifolds by unstable ones. 
\end{itemize}

With items 1 and 2, items $(i)$ and $(ii)$ imply
that for large $k\in\NN$, there is
\begin{itemize}
\item[(iii)] a fundamental domain $\cD_k^+$ in $W^+(C_{\cA_k\cB_k})$ such that the closed set $\overline{\cup_{n\in \NN}C_{\cA_k\cB_k}^{-n}(\cD^+_k)}$ does not intersect the closure of $U$,
\item[(iv)] a fundamental domain $\cD^-_k$ in  $W^-(C_{\cA_k\cB_k})$ such that the closed set $\overline{\cup_{n\in \NN}C_{\cA_k\cB_k}^{n}(\cD^-_k)}$ does not intersect the closure of $U$.  
\end{itemize}
With items 4 and 5, this implies that, for large $k\in\NN$ and all $i\in I$, we have
\begin{align*}
W^{ss,i}(\cA_k)\setminus U&=W^{ss,i}(C_{\cA_k\cB_k})\setminus U,\\
W^{uu,j}(\cA_k)\setminus U&=W^{uu,j}(C_{\cA_k\cB_k})\setminus U.
\end{align*}
Hence, for large $k\in \NN$, $C_{\cA_k\cB_k}$ is an $(I,J)$-connection from $\cA_k$ to $\cB_k$. By $C^1$-convergence of $C_{\cA_k\cB_k}$ and $\cA_k$ to $\cA$, and the fact that $C_{\cA_k\cB_k}=\cA_k$ outside $U$, the size of the connections $C_{\cA_k\cB_k}$ tends to $0$. In particular $d_{I,J}(\cA_k\to \cB_k)$ converges to zero. 
\end{proof}

\subsection{Proof of \cref{t.mainsimplestatement}}
The proof of \cref{t.mainsimplestatement}  summarizes as follows: as a consequence of \cref{c.concat?nation,p.equivdistance} we find a connection $C_{\cA_0\cA_1}$ of size arbitrarily close to the radius of the path $\cA_t$. We linearize the diffeomorphism $f$ locally around the orbit of $P$ thanks to \cref{c.linearisationC1}; then into that linear domain, we glue a conjugate by a homothety $C^\lambda_{\cA_0\cA_1}$ of the $(I,J)$-connection $C_{\cA_0\cA_1}$. For small $\lambda>0$, the map $g_\lambda$ thus obtained will satisfy the required properties. 

\bigskip

We recall without a proof the following folklore:

\begin{lemma}\label{l.folkloremetric}
Let $M$ be a manifold and $K\subset M$ a compact subset.
Let $\|.\|_1$ and $\|.\|_2$ be two Riemannian metrics on $M$ such that they coincide on $T_KM$. 
For any $\epsilon>0$ and $C>1$, there exists a neighborhood $U$ of $K$ such that:
 
 if two diffeomorphisms $g,h$ of $M$ leave $K$ invariant,  coincide outside $U$ and are both bounded by $C$ for $\|.\|_1$, then 
\begin{align*}
\left|\dist_{\|.\|_1}(g,h)-\dist_{\|.\|_2}(g,h)\right|<\epsilon.
\end{align*}
where $\dist_{\|.\|_*}(g,h)$ is the $C^1$-distance between $g$ and $h$.
\end{lemma}

Let $P$ be a periodic point for a diffeomorphism $f$ of a Riemannian manifold $(M,\|.\|)$. We put a linear structure on  a neighborhood of the orbit $\Orb_P$ of $P$ by identifying it to a neighborhood of $0_\cE=\ZZ/p\ZZ\times \{0\}$ in $\cE$ through a diffeomorphism $\phi$ such that $\phi(f^n(P))=(n,0)$, for each integer $n\in \ZZ/p\ZZ$. Choose $\phi$ so that the pull-back $\phi^*\|.\|_{\mbox{Eucl.}}$ of the canonical Euclidean metric $\|.\|_{\mbox{Eucl.}}$ coincices with the metric $\|.\|$ by restriction to $T_{\Orb_P}M$. We extend $\phi^*\|.\|_{\mbox{Eucl.}}$ to some Riemannian metric $\|.\|_1$ on $M$.

The set $\mathfrak{A}$ of linear cycles on $\cE$ identifies through $\phi$ to the set of $p$-uples of linear isomorphisms
$$(A_1,...,A_p)\in \prod_{1\leq n \leq p}\cL\Bigl(T_{f^{n-1}(P)}M,T_{f^n(P)}M\Bigr),$$
where $\cL(E,F)$ is the set of linear isomorphisms from the vector space $E$ to the vector space $F$.
We are now ready for the proof of the main theorem.

\begin{proof}[Proof of \cref{t.mainsimplestatement}] We decompose our proof into three steps.

\bigskip

\noindent {\em Step 1: existence of a connection of the wanted size.}
By the identification above, the radius of the path $\cA_t$ in the hypotheses of  \cref{t.mainsimplestatement} is equal to
\begin{align*}
R=\max_{0\leq t\leq 1}\dist_{\mathfrak{A}}(\cA_0,\cA_t).
\end{align*} 
Let $I$ (resp. $J$) be the set of integers $i>0$ such that, for all $t\in[0,1]$, the linear endomorphism $B_t=A_{p,t}\circ ... \circ A_{1,t}$ admits an $i$-strong stable (resp. unstable) direction.

Let $\epsilon>0$ and $\delta =R+\epsilon$. By \cref{p.equivdistance}, there is a sequence $$t_1=0<t_2<... <t_{\ell-1}<t_\ell=1$$ such that $d_{I,J}(\cA_{t_n},\cA_{t_{n+1}})<\epsilon/4$, for all $1\leq n<\ell$. Thus one finds for each such $n$ an $(I,J)$-connection $C_{\cA_{t_n}\cA_{t_{n+1}}}$ from $\cA_{t_n}$ to $\cA_{t_{n+1}}$ whose size is less than $\epsilon/6$. By \cref{c.concat?nation}, there is an $(I,J)$-connection $C_{\cA_0\cA_1}$ from $\cA_0$ to $\cA_1$ such that 
\begin{align*}
\size(C_{\cA_0\cA_1})&<\max\bigl\{\size(C_{\cA_{t_n},\cA_{t_{n+1}}})+\dist_{\mathfrak{A}}(\cA_{0},\cA_{t_{n}})+\epsilon/6\bigr\}\\
&< R+\epsilon/3.
\end{align*}

\bigskip

\noindent {\em Step 2: local linearization.}
Let $W^\pm(f)$ be local stable and unstable manifolds of $\Orb_P$ for $f$.
By \cref{c.linearisationC1}, there is $\tilde{f}\in\Diff^1(M)$ such that $\dist_{\|.\|}(f,\tilde{f})<\epsilon/3$ and local stable and unstable manifolds $W^\pm(\tilde{f})$ for $\tilde{f}$ such that:
\begin{itemize}
\item $f^{\pm 1}=\tilde{f}^{\pm 1}$ throughout $\Orb_P$ and outside $U$, 
\item $\tilde{f}$ coincides on a neighborhood of $\Orb_P$ with the linear part $\cA_0$ of $f$ along $\Orb_P$,
\item For any $i\in I$, $W^{ss,i}(\tilde{f})$ exists and
$$\left[W^+(f)\cap W^{ss,i}(f)\right]\setminus U=\left[W^+(\tilde{f})\cap W^{ss,i}(\tilde{f})\right]\setminus U,$$
and likewise, replacing stable manifolds by unstable ones. \end{itemize}

\bigskip

\noindent {\em Step 3: plugging the connection into $\tilde{f}$.}
We now glue the $(I,J)$-connection $C^\lambda_{\cA_0\cA_1}$, for some small $\lambda>0$, into the linear region of $\tilde{f}$. Let $V$ be a neighborhood of $\Orb_P$ on which $\tilde{f}$ coincides with the linear diffeomorphism $\cA_0$.\footnote{Recall that $V$ is identified to a neighborhood of $0_\cE$ in $\cE$, through $\phi$. }
Define a map $g_\lambda$ of $M$ by
\begin{align*}
g_\lambda&=C^\lambda_{\cA_0\cA_1} \mbox{ by restriction to $V$}\\
&=\tilde{f}\quad \mbox{ outside $V$.}
\end{align*}

\begin{claim}\label{c.paavap}
For small $\lambda>0$, the strong stable manifolds of $g_\lambda$ and $\tilde{f}$ coincide semilocally outside $U$: for small $\lambda>0$, there exist a a local strong stable manifold $W^{+}(g_\lambda)$ such that, for all $i\in I$, we have
\begin{align*}
\left[W^+(g_\lambda)\cap W^{ss,i}(g_\lambda)\right]\setminus U&=\left[W^+(\tilde{f})\cap W^{ss,i}(\tilde{f})\right]\setminus U.\\
&=\left[W^+(f)\cap W^{ss,i}(f)\right]\setminus U.
\end{align*}
\end{claim}

We leave the proof of \cref{c.paavap} to the reader as it follows exactly the proof of \cref{c.fz}, in \cref{l.concat}. The symmetrical claim obviously holds for strong unstable manifolds.

We are left to estimate the size of the perturbation $g_\lambda$ of $f$. Notice that, since $g_\lambda^{\pm 1}=\tilde{f}^{\pm 1}$ outside the region of $M$ on which $\|.\|_1$ differs from the pull-back of $\|.\|_{\mbox{Eucl.}}$,  we have 
\begin{align*}
\dist_{\|.\|_1}(g_\lambda,\tilde{f})&=\dist_{\|.\|_{\mbox{Eucl.}}}(C^\lambda_{\cA_0\cA_1},\cA_0)\\
&=\size(C^\lambda_{\cA_0\cA_1})\\
&< R+\epsilon/3.
\end{align*}
By \cref{l.folkloremetric}, we deduce that for small $\lambda>0$, $\dist_{\|.\|}(g_\lambda,\tilde{f})<R+2\epsilon/3$. Therefore $\dist_{\|.\|}(g_\lambda,f)<\delta$.

Hence the diffeomorphism $g=g_\lambda$ satisfies all the required properties, for small $\lambda>0$. This ends the proof of \cref{t.mainsimplestatement}.
\end{proof}

\begin{proof}[Sketch of a proof of \cref{t.mainsimplestatementmoregeneral}]
A proof of \cref{t.mainsimplestatementmoregeneral} can easily be adapted from this proof of \cref{t.mainsimplestatement}. Put 
$$\Gamma=\bigl[W^+(g_\lambda)\cap W^{ss,i}(g_\lambda)\bigr]\setminus U=\bigl[W^+(f)\cap W^{ss,i}(f)\bigr]\setminus U.$$
Then the map 
\begin{align*}
\begin{cases}
\Gamma& \to ]0,+\infty[\\
x&\mapsto \dist_{W^{ss,i}(g_\lambda)}(x,\Orb_P)
\end{cases}
\end{align*}
converges uniformly, as $\lambda\to 0$, to the map 
 \begin{align*}
\begin{cases}
\Gamma&\to ]0,+\infty[\\
x&\mapsto \dist_{W^{ss,i}(f_k)}(x,\Orb_P)
\end{cases}.
\end{align*}
Then just take a local stable manifold $W^+(f)$ so that $W^+(f)\setminus U$ contains the sets $K_i$, for all $i\in I$, and we are done. The same obviously holds on the unstable manifolds.
\end{proof}

\section{Proof of \cref{p.pertpropsimple}}\label{s.prop4} In the following, $P$ is a $p$-periodic point for a diffeomorphism $f\in \Diff^r(M)$.

\subsection{Notations and preliminaries}\label{s.eapro}
While the diffeomorphisms we will consider in the following may vary, they will all coincide with $f$ throughout the orbit $\Orb_P$ of the periodic point $P$, and all stable or unstable manifolds of this paper will be those of that orbit. Thus, whenever it exists, we can unambiguously denote the $i$-dimensional strong stable (resp. unstable) manifold of the orbit $\Orb_P$ of $P$ for a diffeomorphism $g$ simply by $W^{ss,i}(g)$ (resp. $W^{uu,i}(g)$). 
 
Let $W^+(g)$ be a local strong stable manifold of $\Orb_P$ for $g$. If $W^{ss,i}(g)$ exists, then the local $i$-strong stable manifold in $W^+(g)$ is defined by
\begin{align*}
W^{+,i}(g)=W^+(g)\cap W^{ss,i}(g)
\end{align*}

Let $I$ be the set of indices $i$ such that $W^{ss,i}(f)$ exists. The diffeomorphisms $g$ we will consider will moreover be $C^r$-close to $f$, in particular $W^{ss,i}(g)$ will exist for all $i\in I$. For simplicity we write 
 $$W^{+,I}(g)=\{W^{+,i}(g)\}_{i\in I}.$$ Furthermore, given a map $\Phi \colon M \to M$, we denote by  $\Phi W^{+,i}(g)$ the image by $\Phi$ of $W^{+,i}(g)$ and denote by 
$$\Phi W^{+,I}(g)=\{\Phi W^{+,i}(g)\}_{i\in I}$$
the incomplete flag of $I$-strong stable manifolds in $W^+(g)$.
We write that {\em $W^{+,I}(g)=W^{+,I}(h)$ outside $U$} (resp. {\em by restriction to $K$}) if fo all $i\in I$, 
\begin{align*}
&W^{+,i}(g)\setminus U=W^{+,i}(h)\setminus U\\
\mbox{(resp. }\;\;\;& W^{+,i}(g)\cap K=W^{+,i}(h) \cap K \;\;). 
\end{align*}
The following lemma could be written in a simpler way, but we chose to state it so that it can  be  used for the proof \cref{l.pqojef} without changes of indices.

\begin{lemma}\label{l.A}
Let $g\in \Diff^r(M)$ be a diffeomorphism that coincides with $f$ throughout $\Orb_P$, and let $W^+(g)$ be a local strong stable manifold for $g$. Let $B\subset M$ be a compact set such that 
\begin{align}
& g^{-2}B, g^{-1}B, B, g(B),..., g^N(B),g^{N+3}(B)  \mbox{ are pairwise disjoint,}\label{s.pwdisjointA}\\
& B \mbox{ is disjoint from $gW^+(g)$.}\label{s.disjointBA}
\end{align}  
Let $\Phi\in \Diff^r(M)$ be a diffeomorphism such that $\Phi=\Id$ outside $B$.  
Define $\Psi,h\in\Diff^r(M)$ by 
\begin{itemize}
\item $\Psi=g^n\circ \Phi \circ g^{-n}$ on $g^nB$, for all $-2\leq n\leq N+2$, and $\Psi=\Id$ outside.
\item $h=g\circ \Psi_k^{-1}$ on $g^{N+2}B$ and $h=g$ outside.
\end{itemize}
Then, $W^+(h)=\Psi W^+(g)$ is a local strong stable manifold for $h$, and 
\begin{align*}
h^{-2}W^{+,I}(h)=\Psi [g^{-2}W^{+,I}(g)].
\end{align*}
\end{lemma}

\begin{proof}
By (\ref{s.disjointBA}), we find a neighborhood $\cV$ of $g^{-2}W^+(g)$ such that 
\begin{align}
g(\cV)\cap g^{-2}(B)=\emptyset\label{e.rergqre}
\end{align} Let 
\begin{align*}
\Gamma&=\cV\cup \bigcup_{-2\leq n\leq N+2}g^{n}(B).
\end{align*}
Notice that $\Psi(\Gamma)=\Gamma$. Notice also that, as $\Psi$ restricts to a diffeomorphism of $g^{N+2}(B)$, $h=g\circ \Psi^{-1}$ on $g^{N+2}(B)$ implies that $h\circ \Psi=g$ on $g^{N+2}(B)$.
\begin{claim}
We have a commuting diagram 
$$\xymatrix{\Gamma \ar[rr]^g \ar[d]^\Psi && g(\Gamma) \ar[d]^\Psi \\
            \Gamma\ar[rr]^h && h(\Gamma)},$$
            that is, $h\circ \Psi (x)=\Psi\circ g(x)$ for all $x\in \Gamma$. 
\end{claim}

\begin{proof}
Let $x\in \Gamma$. There are three cases:
\begin{itemize}
\item If  $x\in g^{n}(B)$, for some $-2\leq n< N+2$, then we have
\begin{align*}
h\!\circ\! \Psi(x)&=g\!\circ\! \Psi(x), \mbox{ since $\Psi(x)\notin g^{N+2}(B)$,}\\
&=g^{(n+1)}\circ \Phi\circ g^{-(n+1)}(g(x))\\
&=\Psi\circ g(x).
\end{align*}
\item If $x\in g^{N+2}(B)$, then 
\begin{align*}
h\circ \Psi(x)=g(x)&\in g^{N+3}(B)\\
&\notin g^{n}(B),  \mbox{ for $-2\leq n\leq N+2$, by  (\ref{s.pwdisjointA})}.\\
\end{align*}
Hence, $\Psi\circ g(x)=g(x)=h\circ \Psi(x)$.
\item If $x\in \cV\setminus \bigcup_{-2\leq n\leq N+2}g^{n}(B)$, then we have
\begin{align*}
x&=\Psi(x),\\
g(x)&=h(x),\\
g(x)&\notin  \bigcup_{-1\leq n\leq N+3} g^{n} (B),\\
g(x)&\notin g^{-2}(B), \qquad\mbox{ by (\ref{e.rergqre})}.
\end{align*}
This implies that $h\circ \Psi(x)=h(x)=g(x)=\Psi\circ g(x)$.
\end{itemize}
We have the wanted equality in all cases: we proved the claim.
\end{proof}

By the claim, the dynamics of $g$ is conjugated to that of $h$, by restriction to the neighborhood $\Gamma$ of $\Orb_P$, through the diffeomorphism $\Psi$. Hence, $W^+(g)\subset \Gamma$ being a local strong stable manifold for $g$, its image $W^+(h)=\Psi W^+(g)$ by $\Psi$ is also one for $h$. 
Moreover, $g^{-2}W^+(g)\subset \Gamma$, easily implies that $h^{-2}W^+(h)=\Psi\bigl[g^{-2}W^+(g)\bigr]$ and $h^{-2}W^{+,i}(h)=\Psi\bigl[g^{-2}W^{+,i}(g)\bigr]$, for all $i\in I$. This ends the proof of the lemma.
\end{proof}

\bigskip

From now on, fix local strong stable and unstable manifolds $W^+(f)$ and $W^-(f)$ for $f$, and a neighborhood $U$ of $\Orb_P$.

In the hypotheses of \cref{p.pertpropsimple}, we may reduce the neighborhood $U$ of $\Orb_P$ and assume that it is open, without loss of generality. Take $N\geq 2$ such that 
$$f^{N}W^{+}(f), f^{-N}W^{-}(f) \subset U.$$
 Let $A^s$ be an annulus, compact thickening of the union of $N+1$ fundamental domains
 $$D^s=f^{-1}W^{+}(f)\setminus f^{N}W^+(f),$$
 that is, $A^s$ identifies through a diffeomorphism to $D^s\times[-1,1]$, with $D^s\simeq D^s\times \{0\}$.
Symmetrically, let $A^u$ be a compact thickening of the union of 
$$D^u=fW^{-}(f)\setminus f^{-N}W^-(f). $$

In particular $A^s\cup U$ and $A^u\cup U$  are neighborhoods of $W^+(f)$ and $W^-(f)$, respectively.
Choosing the thickenings $A^s,A^u$ thin enough, we have that
\begin{itemize}
\item[A.] the sets $A^s$,$A^u$,  $f^{N+2}(A^s)$, $f^{-N-2}(A^u)$ are pairwise disjoint.
\item[B.]  $f^{N+2}(A^s)$ is included in $U$ and disjoint from $W^-(f)$.
\item[C.] $f^{-N-2}(A^u)$ is included in $U$ and disjoint from  $W^+(f)$.
\end{itemize}

\begin{claim}\label{c.pqeriogg}
It is enough to prove \cref{p.pertpropsimple} when $h_k^{\pm 1}=f^{\pm 1}$ on $A^s, A^u$.
\end{claim}

We state without a proof the following folklore:

\begin{lemma}\label{l.qerg}
Given a compact set $K\subset W^s(f)\setminus \Orb_P$ there exists a compact set $L$ that contains $K$ in its interior such that
\begin{itemize}
\item for any sequence of diffeormophisms $h_k$ converging to $f$ in $\Diff^r(M)$, 
\item for any open set $\cO\supset L$, 
\end{itemize}
there exists a sequence $\Phi_k\in \Diff^r(M)$ converging to $\Id_M$  such that for large $k$, we have $\Phi_k=\Id$ outside $\cO$ and $\Phi_k \circ h_k\circ \Phi_k^{-1}=f$ by restriction to $L$.
\end{lemma}

\begin{proof}[Proof of \cref{c.pqeriogg}]
Taking the thickenings $A^u$ and $A^s$ of $D^u$ and $D^s$ thin enough and applying twice \cref{l.qerg} to $K=A^s$ and $K=A^u$, one finds a sequence $\Phi_k\in \Diff^r(M)$ converging to $\Id_M$ such that 
\begin{itemize}
\item $\Phi_k=\Id$ on a neighborhood of $\Orb_P$, 
\item if we put 
\begin{align*}
&\tilde{g}_k=\Phi_k \circ g_k\circ \Phi_k^{-1}\\
&\tilde{h}_k=\Phi_k \circ h_k\circ \Phi_k^{-1}\\
&W^{\pm}(\tilde{h}_k)= \Phi_k\bigl[W^{\pm}(h_k)\bigr],
\end{align*}
then, for large $k\in \NN$, we have $\tilde{h}_k^{\pm 1}=f^{\pm 1}$ by restriction to $A^s$ and $A^u$.
\end{itemize}
Applying \cref{p.pertpropsimple} to the pair of sequences $\tilde{g}_k,\tilde{h}_k$, and a neighborhood $\tilde{U}$ of $\Orb_P$ whose closure lies in the interior of $U$, we get sequences $\tilde{f}_k$ and $W^{\pm}(\tilde{f}_k)$. Then, for large $k\in \NN$, the sequences $f_k=\Phi_k^{-1} \circ \tilde{f}_k\circ \Phi_k$ and $W^{\pm}(f_k)= \Phi_k^{-1}\bigl[W^{\pm}(\tilde{f}_k)\bigr]$ will satisfy the conclusions of  \cref{p.pertpropsimple} with respect to $U$ and to the sequences $g_k,h_k,$ and $W^{\pm}(h_k)$. This ends the proof of \cref{c.pqeriogg}.\end{proof}

\subsection{Thick torus of orbits in $A^s$ and flag of invariant manifolds}

The {\em space of maximal segments of orbits in $A^s$} for $f$ is the quotient space $A^s\!/\!\!\sim$, where $x\sim y$ if and only if there is a segment of $f$-orbit included in $A^s$ that joins $x$ to $y$.\footnote{For $A^s$ thin enough, which we will assume, the equivalence classes are indeed segments of orbits.} Thanks to \cref{c.pqeriogg}, all the diffeomorphisms we will consider can be assumed to coincide with $f$ on $A^s$. In particular, the space of maximal segments orbits in $A^s$ will be common to all those dynamics. Let $\pi\colon A^s \to A^s\!/\!\!\sim$ be the canonical projection.

We choose $A^s$ thin enough so that $f^{-2}W^+(f)\cap A^s=D^s$. Then $f^{-2}W^+(f)\cap A^s$ is a union of $\sim$ classes, in particular it identifies to its image by $\pi$ in $A^s\!/\!\!\sim$\footnote{This is not the case of $W^+(f)\cap A^s$, which is made of non-maximal segments of $f$-orbits in $A^s$, hence it has not natural identification to a subset of $A^s\!/\!\!\sim$}.
 We denote that image by $W^+(f)/\!\!\sim$. 
 
 Notice that 
  $$W^+(f)/\!\!\sim\;\;=\;\pi[f^{-1}W^{+}(f)\setminus f^{N}W^+(f)]$$
is diffeomorphic to the "torus" $\SS^1\times \SS^m$, where $m+1$ is the dimension of $W^+(f)$.
One naturally endows a closed neighborhood $T$ of $W^+(f)/\!\!\sim$ in $A^s/\!\!\sim$ with a structure of $C^r$ manifold so that the restriction $\pi : \pi^{-1}(T)\to T$ is a $C^r$ local diffeomorphism. Redefine $A^s$ by $A^s=\pi^{-1}T$. The manifold $T$ is the {\em thick torus of orbits} in the annulus $A^s$.

For simplicity, in the following
\begin{itemize}
\item "$f_k\to f$" will stand for

"a sequence of diffeomorphisms $f_k\in \Diff^r(M)$ converging to $f$ for the $C^r$ topology,"
\item  "$W^+(f_k)\to W^+(f)$" will stand for

"a sequence  $W^+(f_k)$ of local strong stable manifolds for the diffeomorphisms $f_k$ converging to $W^+(f)$ for the $C^r$-topology".
\end{itemize}

\begin{lemma}\label{l.tori}
Let $f_k\to f$ and $W^+(f_k)\to W^+(f)$, with $f_k^{\pm 1}=f^{\pm 1}$ on $A^s$, in particular the segments of orbits in $A^s$ are the same for $f_k$ and $f$.

Then for large $k\in\NN$, the sets $f_k^{-2}W^+(f_k)\cap A^s$ and $f_k^{-2}W^{+,i}(f_k)\cap A^s$, for all $i\in I$, are made of maximal segments of $f=f_k$-orbits in $A^s$. Therefore, they identify to their respective images by $\pi$.
\end{lemma}

We denote
 $$W^+(f_k)/\!\!\sim\;\;=\;\pi\bigl[f_k^{-2}W^+(f_k)\cap A^s\bigr],$$ 
$$W^{+,i}(f_k)/\!\!\sim \;\; = \; \pi(f_k^{-2}W^{+,i}(f_k)\cap A^s).$$

\begin{remark}\label{r.tori}
For large $k\in\NN$, the sets $W^+(f_k)/\!\!\sim$ are tori in $T$, and they $C^r$-converge to the torus $W^+(f)/\!\!\sim$ as $k\to \infty$. Likewise, for large $k$, the sets $W^{+,I}(f_k)/\!\!\sim$ are $C^r$-submanifolds of $T$ diffeomorphic to $\SS^1\times \SS^{i-1}$, and they $C^r$-converge to the submanifold $W^{+,i}(f)/\!\!\sim$.
\end{remark}
%
%The sets $W^{+,i}(f_k)$ form a sequence $W^{+,I}(f_k)/\!\!\sim$ of flags of manifolds that converges $C^r$ to the flag $W^{+,I}(f)/\!\!\sim$.
%\end{remark}
We say that the sequence of (incomplete) flag of manifolds  
$$W^{+,I}(f_k)/\!\!\sim=\bigl\{W^{+,i}(f_k)/\!\!\sim\bigr\}_{i\in I}$$
$C^r$-converges to the flag $W^{+,I}(f)/\!\!\sim$ within $T$.

\begin{proof}[Proof of \cref{l.tori}]
Let $x_k\in f_k^{-2}W^+(f_k)\cap A^s$ be a sequence converging to some $x\in D_n= f^{-2}W^+(f)\cap A^s$. Let $(f^nx)_{p\leq n\leq q}$, where $p\leq 0\leq q$, be the maximal segment of $f$-orbit through $x$ in $A^s$. As $f^nx\in \Int A^s$ for $p<n<q$ and by compactness of $A^s$, for large $k$, the maximal segment of orbit through $x_k$ in $A^s$ is of the form  $(f^nx)_{p_k\leq n\leq q_k}$, where $p_k\in \{p,p+1\}$ and $q_k\in \{q-1,q\}$. As the points $(f^nx)_{p\leq n\leq q}$ lie in the interior of $f^{-2}W^+(f)$, the points $(f^nx)_{p_k\leq n\leq q_k}$ lie in $f_k^{-2}W^+(f_k)$, for large $k$. In particular, $f_k^{-2}W^+(f_k)\cap A^s$ contains the maximal segment of orbit  maximal segment of $f$-orbit through $x_k$ in $A^s$. By a compactness argument, one deduces that, for large $k$, the set $f_k^{-2}W^+(f_k)\cap A^s$ is a union of maximal segments of $f$-orbits in $A^s$.
\end{proof}

Let $x\in \Int\bigl[f^{-1}W^{+}(f)\bigr]\setminus fW^+(f)$. If a ball $B$ centered at $x$ is small  enough, then it clearly satisfies:
\begin{itemize}
\item[$(i)$] the restriction $\pi_{|B}$ of $\pi$ to $B$ is a diffeomorphism.
\item[$(ii)$] For all $0\leq n< N$, $f^nB\subset \Int A^s$. Indeed, $f^n(x)\in\Int D^s$.
\item[$(iii)$] $f^{-2}B\cap A^s=\emptyset$ and $f^{N+1}B\cap A^s=\emptyset$. Indeed $f^{-2}x,f^{N+1}x\notin D^s$.
\item[$(iv)$] From $(ii)$ and $N\geq 2$, we have $f^{N+2}B,f^{N+3}B\subset \Int( f^{N+2}A^s)$.
\item[$(v)$] From $(ii)$ and $(iii)$, we get the following: for any $z\in M$ and $y\in B$, we have $z\sim y$ if and only if $z\in A^s$ and $z=f^n(y)$, for some $-1\leq n\leq N$.  In other words,
\begin{align*}\pi^{-1}\circ \pi (B)&=\bigcup_{-1\leq n\leq N}A^s\cap f^n(B).\\
&=\bigcup_{-2\leq n\leq N+2}A^s\cap f^n(B),\mbox{ by $(iii)$, $(iv)$ and item A.}
\end{align*}
\item[$(vi)$] The sets $f^{-2}B,f^{-1}B,....,f^{N+3}B$ are pairwise disjoint.
\item[$(vii)$]  $B\cap fW^+(f)=\emptyset$.
\end{itemize}
By a compactness argument, one finds $\ell$ such balls $B_1,...,B_\ell$ such that the interiors $\Int B_i$ cover a fundamental domain of $W^+(f)$, hence the open sets $\pi(\Int B_i)$ cover the torus $W^+(f)/\!\!\sim$.

\begin{lemma}\label{l.pqojef}
Let $f_k\to f$ and $W^+(f_k)\to W^+(f)$ such that $f^{\pm 1}=f_k^{\pm 1}$ on $A^s$. This gives a sequence  $\cF_k=W^{+,I}(f_k)/\!\!\sim$ of flags of $T$, as defined in \cref{r.tori}. Let $B=B_i$ be one of the balls above, and let $\Phi_k$ be a sequence in $\Diff^r(M)$ converging to $\Id$, such that $\Phi_k=\Id$ outside $\pi(B)$. 

Then there exist sequences $g_k\to f$  and $W^+(g_k)\to W^+(f)$, such that 
\begin{itemize} 
\item $g_k=f_k$ outside the set $f_k^{N+2}B$, hence $g_k^{-1}=f_k^{-1}$ outside $f_k^{N+3}B$.
\item for large $k\in \NN$, we have $g_k^{\pm 1}=f_k^{\pm 1}=f^{\pm 1}$ on $A^s$, hence the flag $W^{+,I}(g_k)/\!\!\sim$ is well-defined, by  \cref{l.tori,r.tori},
\item moreover, $W^{+,I}(g_k)/\!\!\sim$ is the image $\Phi_k\cF_k$ by $\Phi_k$ of the flag $\cF_k$.
\end{itemize}
\end{lemma}

\begin{remark}\label{r.porefjap}
In particular, the conclusions imply that for large $k$, $g_k^{\pm 1}=f_k^{\pm 1}$ outside $f^{N+2}(A^s)$, since $(i)$ implies that $f_k^{N+2}B,f_k^{N+3}B\subset \Int f^{N+2}(A^s)$, for large $k\in \NN$.
\end{remark}

\begin{proof}[Proof of \cref{l.pqojef}]

Let $\tilde{\Phi}_k\in \Diff^r(M)\to \Id$ such that  $\tilde{\Phi}_k=\Id$ outside $B$ and such that the following diagram commutes:
$$\xymatrix{B \ar[rr]^{\tilde{\Phi}_k} \ar[d]^\pi && B \ar[d]^\pi \\
            \pi(B)\ar[rr]^{\Phi_k} && \pi(B)}$$
By $(vi)$ and $(vii)$, the compact sets $f_k^{-2}B,f_k^{-1}B,....,f_k^{N+3}B$ are pairwise disjoint and $B\cap f_kW^+(f_k)=\emptyset$, for large $k$. Define $\Psi_k,g_k\in\Diff^r(M)$ by 
\begin{itemize}
\item $\Psi_k=f_k^n\circ \tilde{\Phi}_k \circ f_k^{-n}$ on $f_k^nB$, for all $-2\leq n\leq N+2$, and $\Psi_k=\Id$ outside.
\item $g_k=f_k\circ \Psi_k^{-1}$ on $f_k^{N+2}B$ and $g_k=f_k$ outside.
\end{itemize}
By \cref{l.A}, $W^+(g_k)=\Psi_k W^+(f_k)$ is a local strong stable manifold for $g_k$,
and $g_k^{-2}W^{+,I}(g_k)=\Psi_k \bigl[f_k^{-2}W^{+,I}(f_k)\bigr]$. In particular, we have 
\begin{align}
g_k^{-2}W^{+,I} \cap B=\Psi_k \bigl[f_k^{-2}W^{+,I}(f_k)\bigr]\cap B .\label{e.pzogj}
\end{align}
By $(iv)$ and item A of \cref{s.eapro},  for large $k\in \NN$, $f_k^{N+2}B$ and $f_k^{N+3}B$ do not intersect $A^s$, thus we have $g_k^{\pm 1}=f_k^{\pm 1}=f^{\pm 1}$ on $A^s$. 

On the other hand, \cref{l.tori} implies that, for $\xi=f$ or $g$ and for large $k\in \NN$, we have
\begin{align}
\pi\bigl[\xi_k^{-2}W^{+,I}(\xi_k) \cap B\bigr]&=\pi\bigl[\xi_k^{-2}W^{+,I}(\xi_k)\bigr] \cap \pi(B)\nonumber\\
&=\left[W^{+,I}(\xi_k)/\!\!\sim\right] \cap \pi(B).\label{e.pzogj2}
\end{align}
With the commuting diagram
$$\xymatrix{B \ar[rr]^{\Psi_k} \ar[d]^\pi && B \ar[d]^\pi \\
            \pi(B)\ar[rr]^{\Phi_k} && \pi(B)}$$
we deduce from \cref{e.pzogj,e.pzogj2} that 
$$\bigl[W^{+,I}(g_k)/\!\!\sim\bigr]\cap \pi(B)=\;\Phi_k\bigl[W^{+,I}(f_k)/\!\!\sim \bigr]\cap \pi(B).$$

By $(ii)$, we have $f_k^{n}(B)=f^n(B)$, for all $-1\leq n\leq N$. By $(iii)$, $(iv)$ and item A of \cref{s.eapro}, we have $f_k^{n}(B)\cap A^s=\emptyset$, for $n=-2,N+1,N+2$. Hence with $(v)$, we deduce that 
\begin{align*}\pi^{-1}\circ \pi (B)&=\bigcup_{-2\leq n\leq N+2}A^s\cap f_k^n(B)\end{align*}
Therefore, by definition of $\Psi_k$, we have  $\Psi_k=\Id$ on $A^s \setminus \pi^{-1}\circ \pi(B)$, hence $W^{+,I}(g_k)/\!\!\sim\;\;=W^{+,I}(f_k)/\!\!\sim$ outside $\pi(B)$, which ends the proof of the lemma.
\footnote{$\Psi_k$ does not restrict to  a diffeomorphism of $A^s$, but to a diffeomorphism of $A^s\setminus \pi^{-1}\circ \pi(B)$, and to a diffeomorphism of $B$. That is why we cannot write the following commuting diagram:
\begin{align*}
\xymatrix{A^s \ar[rr]^{\Psi_k} \ar[d]^\pi && A^s \ar[d]^\pi \\
           T\ar[rr]^{\Phi_k} && T}
          \end{align*}}
\end{proof}

\begin{lemma}[Sequential fragmentation lemma]
Let $B_1,...,B_\ell$ be closed sets and let $K$ be a compact set such that 
$$K\subset \bigcup_{1\leq s\leq \ell}\Int (B_s).$$
Let $\Phi_k\in \Diff^r(T)$ be a sequence of diffeomorphisms converging to $\Id$ such that $\Phi_k=\Id$ outside $K$. Then there exists $\Phi_{s,k}\to \Id$ in $\Diff^r(M)$, for each $1\leq s\leq \ell$, such that $\Phi_{s,k}=\Id$ outside $B_s$, and for large $k$,
\begin{align}
\Phi_k=\Phi_{\ell,k}\circ ...\circ \Phi_{1,k}\label{e.rgqrg}
\end{align}
\end{lemma}

\begin{proof}
Fix a family of closed sets $C_s\subset \Int B_s$ that covers $K$. We build the sequences $\Phi_{s,k}$ by induction on $s$ so that they satisfy the following additional hypotheses, for all $1\leq t\leq \ell$ and for large $k\in \NN$:
\begin{itemize}
\item $\Phi_k=\Phi_{t,k}\circ ...\circ \Phi_{1,k}$ by restriction to $C_1\cup ...\cup C_t$, 
\item $\Phi_{t,k}\circ ...\circ \Phi_{1,k}=\Id$ outside $K$. 
\end{itemize}
Notice that those two items at rank $\ell$ trivially imply \cref{e.rgqrg}.
Assume this is satisfied at rank $t$ by some sequences $\Phi_{s,k}$, for $1\leq s\leq t$. The sequence $\xi_k=\Phi_k\circ \Phi_{1,k}^{-1}\circ ...\circ \Phi_{t,k}^{-1}$ converges to $\Id$ in $\Diff^r(M)$. For large $k\in \NN$, we have $\xi_k=\Id$ by restriction to $C_1\cup ...\cup C_t$ and outside $K$. By a partition of unity, we build a sequence $\Phi_{t+1,k}$ of diffeomorphisms in $\Diff^r(M)$ converging to $\Id$ such that 
\begin{align*}\Phi_{t+1,k}&=\xi_k \mbox{ on $C_{t+1}$,}\\
\mbox{ and }\Phi_{t+1,k}&=\Id\mbox{ where $\xi_k=\Id$ and outside $B_{t+1}$.}
\end{align*} The two items above are again satisfied at rank $t+1$. This ends the proof by induction.
\end{proof}

\begin{corollary}
Let $f_k\to f$ and $W^+(f_k)\to W^+(f)$ such that $f^{\pm 1}=f_k^{\pm 1}$ on $A^s$. Let $\cG_k$ be a sequence of flags of manifolds in $T$ converging to $\cF=W^{+,I}(f)/\!\!\sim$ for the $C^r$-topology, that is, $\cG_k=\{\cG_{k,i}\}_{i\in I}$ is a flag of manifolds, for all $k$, and  $\cG_{k,i}\to W^{+,i}(f)/\!\!\sim$, for all $i$.

 Then there are sequences $g_k\to f$ and $W^+(g_k)\to W^+(f)$, such that $g_k^{\pm 1}=f_k^{\pm 1}$ outside $f^{N+2}(A^s)$, and such that $\cG_k=W^{+,I}(g_k)/\!\!\sim$.
\end{corollary}

\begin{proof}
Let $K$ be  a compact neighborhood of $W^+(f)/\!\!\sim$ contained in the union $\bigcup_{1\leq s\leq \ell}\Int (B_s)$. Let $\cF_k=W^{+,I}(f_k)/\!\!\sim$. The sequences of flags $\cF_k$ and $\cG_k$ both tend to $\cF$ for the $C^r$ topology. It is a folklore differential topology result that there exists a sequence $\Phi_k\in \Diff^r(T)$ that converges to $\Id$ such that $\cG_k=\Phi_k\cF_k$ and $\Phi_k=\Id$ outside $K$, for large $k\in \NN$. Let $\{\Phi_{s,k}\}_{1\leq s\leq \ell}$ be sequences of diffeomorphisms of $T$ given by the  sequential fragmentation lemma.

By \cref{l.pqojef}, we build inductively sequences $g_{j,k}$ and $W^+(g_{j,k})$  such that $g_{0,k}=f$,
$$W^+(g_{s,k})/\!\!\sim \;=\Phi_{j,k}\circ ... \circ \Phi_{1,k}(\cF_k),$$
and with \cref{r.porefjap}, for large $k$, $g_{s,k}^{\pm 1}=f_k^{\pm 1}$ outside $f^{N+2}(A^s)$, in particular $g_{s,k}^{\pm 1}=f^{\pm 1}$ on $A^s$. Take $g_k=g_{\ell, k}$ and we are done.
\end{proof}

\subsection{Proof of  \cref{p.pertpropsimple}}

By a partition of unity, one builds a sequence $\hat{f}_k\to f$ and finds a neighborhood $V$ of $\Orb_P$ such that 
\begin{itemize}
\item $\hat{f}_k^{\pm 1}=g_k^{\pm 1}$ on $V$,
\item $\hat{f}_k^{\pm 1}=h_k^{\pm 1}$ on $A^s,A^u$ and outside $U$, in particular we may assume that $\hat{f}_k^{\pm 1}=f^{\pm 1}$ on $A^s, A^u$, by \cref{c.pqeriogg}.
\end{itemize}

\begin{proposition}\label{p.fs}
There exists $f_k^s\to f$, $W^+(f_k^s)\to W^+(f)$ such that ${f_k^s}^{\pm 1} =  \hat{f}^{\pm 1}$ outside $f^{N+2}(A^s)$ and $W^{+,I}(f^s_k)=W^{+,I}(h^s_k)$ outside $U$. 
\end{proposition}

This proposition is proved at the end of this section. Replacing $f$ by $f^{-1}$, one deduces straightforwardly the symmetrical proposition:

\begin{proposition}\label{p.fu}
There exists $f_k^u\to f$, $W^-(f_k^u)\to W^-(f)$ such that ${f_k^u}^{\pm 1} =  \hat{f}^{\pm 1}$ outside $f^{-N-2}(A^u)$ and $W^{-,J}(f^u_k)=W^{-,J}(h^u_k)$ outside $U$. 
\end{proposition}

\begin{proof}[Proof of  \cref{p.pertpropsimple}]
By \cref{p.fs,p.fu} and item A in \cref{s.eapro}, for large $k\in \NN$, there is a well-defined diffeomorphism $f_k\in \Diff^r(M)$ such that 
\begin{align*}
f^{\pm 1}_k&=\hat{f}^{\pm 1}_k \mbox{ outside $f^{N+2}(A^s)$ and $f^{-N-2}(A^u)$} \\
f^{\pm 1}_k&=f^{s \pm 1}_k \mbox{ on $f^{N+2}(A^s)$,} \\
f^{\pm 1}_k&=f^{u \pm 1}_k \mbox{ on $f^{-N-2}(A^u)$} 
\end{align*}
By items B and C in \cref{s.eapro}, $f_k^{\pm 1}=\hat{f}_k^{\pm 1}$ outside $U$ hence $f_k^{\pm 1}=h_k^{\pm 1}$ outside $U$. Moreover, $f_k^{\pm 1}=\hat{f}_k^{\pm 1}=g_k^{\pm 1}$ on $V\setminus \bigl[f^{N+2}(A^s)\cup f^{-N-2}(A^u)\bigr]$, which is a neighborhood of  $\Orb_P$, since $A^s$ and $A^u$ are compact sets disjoint from $\Orb_P$. On the other hand, item B implies that  $f_k$ coincides with $f_k^s$ on a neighborhood of $W^+(f^s_k)$ and item $C$ implies that it coincides with $f_k^u$ on a neighborhood of $W^-(f^u_k)$, for large $k$. In particular, $W^{+}(f_k)=W^{+}(f^s_k)$ is a sequence of local strong stable manifolds, and $W^{+,I}(f_k)=W^{+,I}(f^s_k)$, which coincides with $W^{+,I}(h^s_k)$ outside $U$. The same happens on the unstable side, therefore the conclusions of \cref{p.pertpropsimple} are all satisfied.
\end{proof}

\begin{proof}[Proof of \cref{p.fs}]
Take $\cG_k=W^{+,I}(h_k)/\sim$. Then, by \cref{l.pqojef}, there are sequences $f^s_k\to f$, $W^+(f_k^s)\to W^+(f)$, such that ${f_k^s}^{\pm 1} =  \hat{f}_k^{\pm 1}$ outside $f^{N+2}(A^s)$ and, for large $k$, $W^{+,I}(f_k^s)/\!\!\sim\;=\;\cG_k=W^{+,I}(h_k)/\!\!\sim$. In other words, 
\begin{align}
{(f^s_k)}^{-2}W^{+,I}(f_k^s)=h_k^{-2}W^{+,I}(h_k)\mbox{ by restriction to $A^s$.} \label{e.2star}
\end{align}
As $\partial W^+(h_k)\subset A^s\cap h_k^{-2}W^+(h_k)=A^s\cap {(f_k^s)}^{-2}W^+(f_k^s)$, for large $k$, we may redefine a local strong stable manifold  $W^+(f_k^s)$ for $f_k^s$ as the union of disks delimited by $\partial W^+(h_k)$ in ${(f_k^s)}^{-2}W^+(f_k^s)$. 

Indeed, this new sequence $W^+(f_k^s)$ converges to $W^+(f)$ for the $C^r$-topology, and from the strict $f$-invariance of $W^+(f)$, we deduce that for large $k$, $W^+(f_k^s)\subset W^{ss}(f_k^s)$ is also strictly $f_k^s$-invariant. 

By \cref{e.2star} and by construction, we have $W^{+}(f^s_k)\cap A^s=W^{+}(h_k)\cap A^s$. Again with \cref{e.2star}, we deduce that
 $W^{+,I}(f^s_k)=W^{+,I}(h_k)$ by restriction to $A^s$.
 
 $A^s\cup U$ is a neighborhood of $W^+(f)$, hence $W^{+}(f^s_k), W^{+}(h_k)\subset A^s\cup U$, for large $k$. Therefore $W^{+,I}(f^s_k)=W^{+,I}(h_k)$ outside $U$.
\end{proof}

\section{Examples of applications.}\label{s.consequences}

We give in this section a few consequences of  Theorem~\ref{t.mainsimplestatement}. We prove Theorem~\ref{t.realeigen} which asserts that one can perturb a saddle of large period in order to turn its eigenvalues real, while preserving its invariant manifolds semi-locally.

Wen~\cite{W1} showed that the absence of a dominated splitting of index $i$ on limit sets of periodic orbits of same index allows to create homoclinic tangencies by small perturbations. To prove it, he showed that one obtains new saddles with small stable/unstable angles by $C^1$-pertubations, but a priori without knowledge of the homoclinic class to which the new saddles belong. Here we prove Theorem~\ref{t.dichotomysimple}, which gives a dichotomy between small angles and dominated splittings within homoclinic classes. Through that result, we find another way to the main theorem of~\cite{Gou}, and the more result \cref{t.homoclinictang}.

\subsection{Dichotomy between small angles and dominated splittings.}\label{s.dichotomy}

We recall that a saddle point $P$ is {\em homoclinically related} to another $Q$ if and only if the unstable manifold of each meets somewhere transversally the stable manifold of the other. The {\em homoclinic class} of a saddle $P$ is the closure of the saddles that are homoclinically related to $P$.  The {\em eigenvalues} of a saddle $P$ are the eigenvalues of the derivative of the first return map at $P$.

\begin{definition}\label{d.dominatedsplitting}
Let $f$ be a diffeomorphism of $M$ and $K$ be a compact invariant set. A splitting $TM_{|K}=E\oplus F$ of the tangent bundle above $K$ into two $Df$-invariant vector subbundles of constant dimensions is a {\em dominated splitting} if there exists an integer $N\in \NN$ such that, for any point $x\in K$, for any unit vectors $u\in E_x$ and $v\in F_x$ in the fibers of $E$ and $F$ above $x$, respectively, one has:
$$\|Df^N(u)\|<1/2.\|Df^N(v)\|.$$  
\end{definition}
In that case, we say the splitting is {\em $N$-dominated}. The smaller the number $N$, the stronger the domination. 

Theorem~\ref{t.dichotomysimple} is a generic consequence of the following proposition (see section~\ref{s.dichotomysimple}).

\begin{proposition}
\label{p.smallanglediffeo}
Let $f$ be a diffeomorphism of $M$ and $\epsilon>0$ be a real number. There exists an integer $N\in \NN$ such that for any
\begin{itemize}
\item saddle periodic point $P$ of period $p\geq N$ such that the corresponding stable/unstable splitting is not $N$-dominated,,
\item neighbourhood $U$ of the orbit $\Orb_P$ of $P$, 
\item number $\varrho>0$ and families of compact sets 
\begin{align*}K_i&\subset W^{i,ss}_\varrho(P,f)\setminus \Orb_P, \quad \mbox{ for all }  i\in I\\
 L_j&\subset W^{j,uu}_\varrho(P,f)\setminus \Orb_P,  \quad  \mbox{ for all } j\in J,
 \end{align*} 
 where $I$ and $J$  are the sets of the strong stable and unstable dimensions,
 \end{itemize}
there is a $C^1$-$\epsilon$-perturbation $g$ of $f$ such that
\begin{itemize}
\item  $f^{\pm 1}=g^{\pm 1}$ throughout $\Orb_P$ and outside $U$, 
\item the minimum stable/unstable angle for $g$ of some iterate $g^k(P)$ is less than $\epsilon$,
\item the eigenvalues of the first return map $Dg^p(P)$ are real, pairwise distinct and each of them has modulus less than $\epsilon$ or greater than $\epsilon^{-1}$,
\item  for all $(i,j)\in I\times J$, we have
 $$K_i\subset W^{ss,i}_\varrho(P,g)\mbox{ and }L_j\subset W^{uu,j}_\varrho(P,g).$$
\end{itemize}
\end{proposition}

The proof of Proposition~\ref{p.smallanglediffeo} is postponed until section~\ref{s.pathcocycles}.

Theorem 4.3 in~\cite{Gou} states that if the stable/unstable dominated splitting along a saddle is weak enough, then one may find a $C^1$-perturbation that creates a homoclinic tangency related to that saddle, while preserving a finite number of points in the strong stable/unstable manifolds of that saddle. During the process, the derivative of that saddle may have been modified. The technique introduced in this paper allows to create a tangency while preserving the derivative.

Indeed, under the hypothesis that there is a weak stable/unstable dominated splitting for some saddle $P$, one creates a small stable/unstable angle and pairwise distinct real eigenvalues of moduli less than $1/2$ or greater than $2$, after changing the derivative by  application of Theorem~\ref{t.mainsimplestatement} with some path $\cA_t$ of derivatives (see the proof of Proposition~\ref{p.smallanglediffeo} in section~\ref{s.pathcocycles}). Applying the techniques of the proof of ~\cite[Proposition~5.1]{Gou}, one finds another small $C^1$-perturbation on an arbitrarily small neighbourhood of $P$ that creates a tangency between its stable and unstable manifolds, without modifying the dynamics on a (smaller) neighbourhood of the orbit of $P$. That perturbation can be done preserving any preliminarily fixed finite set inside the strong stable or unstable manifolds of $P$. Then one may come back to the initial derivative applying again Theorem~\ref{t.mainsimplestatement} with the backwards path $\cA_{1-t}$. This sums up into:
 
\begin{theorem}
\label{t.homoclinictang}
Let $f$ be a diffeomorphism of $M$ and $\epsilon>0$ be a real number. There exists an integer $N\in \NN$ such that  if $P$ is a saddle point of period greater than $N$ and its corresponding stable/unstable splitting is not $N$-dominated, if $U$ is a neighbourhood of the orbit of $\Orb_P$ and $\Gamma\subset M$ is a finite set, then
\begin{itemize}
\item  there is a $C^1$ $\epsilon$-perturbation $g$ of $f$ such that $f^{\pm 1}=g^{\pm 1}$ throughout $\Orb_P$ and outside $U$, and such that the saddle $P$ admits a homoclinic tangency inside $U$ for $g$.
\item  the derivatives $Df$ and $Dg$ coincide along the orbit of $P$, 
\item for each $x\in \Gamma$, if $x$ is in the strong stable (resp. unstable) manifold of dimension $i$ of $\Orb_P$ for $f$, then $x$ is also the strong stable (resp. unstable) manifold of dimension $i$ of $\Orb_P$ for $g$.
\end{itemize}
\end{theorem}

\subsection{Proof of Theorem~\ref{t.dichotomysimple}}\label{s.dichotomysimple}
Fix $p\in \NN\setminus{0}$ and $\epsilon>0$. 
Let $\cS_{p,\epsilon}$ be the set of diffeomorphisms $f$ such that for any periodic saddle point $P$ of period $p$, if the homoclinic class of $P$ has no dominated splitting of same index as $P$, then there is a saddle $Q$ in the homoclinic class of $P$ with same index as $P$ that has a minimum stable/unstable angle less than $\epsilon$ and pairwise distinct real eigenvalues of moduli less than $\epsilon$ or greater than $\epsilon^{-1}$.

\begin{lemma}\label{l.smallangles2}
For all $p\in \NN\setminus{0}$ and $\epsilon>0$, the set $\cS_{p,\epsilon}$ contains an open and dense set in $\Diff^1(M)$. 
\end{lemma}

\begin{proof}[Proof of Theorem~\ref{t.dichotomysimple}:]
Take the residual set $\displaystyle \cR=\bigcap_{p,n\in \NN}\cS_{p,\frac{1}{n+1}}.$
\end{proof}

\begin{proof}[Proof of Lemma~\ref{l.smallangles2}:]
By the Kupka-Smale Theorem, there is a residual set $\cR$ of diffeomorphism whose periodic points are all hyperbolic, and consequently that have a finite number of periodic points of period $p$. 
Let $f\in \cR$.   Let $P_1,...,P_l$ be the saddle points of period $p$ for $f$. For all $g$ in some neighborhood $\cU_f$ of $f$, the saddle points of period $p$ for $g$ are the continuations $P_1(g),...,P_l(g)$ of the saddles $P_1,...,P_l$. 

\begin{claim}
For all $1\leq k\leq l$, there is an open and dense subset $\cV_k$ of $\cU_f$ such that, for all $g\in \cV_k$, the homoclinic class of the continuation of $P_k(g)$ either admits a dominated splitting of same index as $P_k$, or contains a saddle of same index as $P_k$ that has a minimum stable/unstable angle less than $\epsilon$ and pairwise distinct real eigenvalues of moduli less than $\epsilon$ or greater than $\epsilon^{-1}$.
\end{claim}

\begin{proof} Let $\Delta\subset \cU_f$ be the set of diffeomorphisms such that the homoclinic class of the continuation $P_k(g)$ does not admit a dominated splitting of same index as $P_k$, and let $\Delta_\epsilon\subset \cU_f$ be the open set of diffeomorphisms such that that homoclinic class contains a saddle of same index as $P_k$ that has a stable/unstable angle strictly less than $\epsilon$ and pairwise distinct real eigenvalues of moduli less than $\epsilon$ or greater than $\epsilon^{-1}$. Let $f\in \Delta$. 

Obviously, the homoclinic class of $P_k(f)$ cannot be reduced to $P_k(f)$. 
For any $N\in \NN$, there is a periodic point $Q_N$ in that homoclinic class that has same index as $P_k$, that has period greater than $N$, and such that the stable/unstable splitting above the orbit of $Q_N$ is not $N$-dominated. 
By Proposition~\ref{p.smallanglediffeo}, there is an arbitrarily small perturbation of $g$  that turns  
 the minimum stable/unstable angle of some iterate of some $Q_N$ to be strictly less than $\epsilon$, and that turns  the eigenvalues of that $Q_N$ to be real, with pairwise distinct with moduli less than $\epsilon$ or greater than $\epsilon^{-1}$, while preserving the dynamics and preserving any previously fixed pair of compact sets $K^u, K^s$ (that do not intersect $\Orb_{Q_N}$) in the stable and unstable manifolds of $Q_N$. In particular, one can do that perturbation preserving the homoclinic relation between $Q_N$ and $P_k(g)$: one finds an arbitrartily small perturbation of $g\in \Delta$ in $\Delta_\epsilon$.

Thus $\Delta^c\cup \cl(\Delta_\epsilon)=\cU_f$, where $\Delta^c=\cU_f\setminus \Delta$. As a consequence, $\Delta^c\setminus \cl(\Delta_\epsilon)$ is open and $$\cV_k=\bigl[\Delta^c\setminus \cl(\Delta_\epsilon)\bigr]\cup \Delta_\epsilon$$ satisfies all the conclusions of the claim.   
\end{proof}
The intersection $\cV_f=\cap_{1\leq k\leq  l} \cV_k$ is an open and dense subset of $\cU_f$ and is included in $\cS_{p,\epsilon}$. The union of such $\cV_f$ is an open and dense subset of $\Diff^1(M)$ contained in $\cS_{p,\epsilon}$. This ends the proof of the Lemma.
\end{proof}

\subsection{Linear cocycles and dominated splittings.} Here we recall notations and tools from~\cite{BDP} and~\cite{BGV}.
Let $\pi\colon E \to \cB$ be a vector bundle of dimension $d$ above a compact base $\cB$ such that, for any point $x\in \cB$, the fiber $E_x$ above $x$ is a $d$-dimensional vector space endowed with a Euclidean metric $\|.\|$. One identifies each $x\in\cB$ with the zero of the corresponding fiber $E_x$. A {\em linear cocycle $\cA$ } on $E$ is a bijection of $E$ that sends each fiber $E_x$ on a fiber by a linear isomorphism. We say that $\cA$ is {\em bounded} by $C>1$, if for any unit vector $v\in E$, we have $C^{-1}<\|\cA(v)\|<C$.

In the following, a {\em subbundle} $F\subset E$, is a vector bundle with same base $\cB$ as $E$ such that, for all $x,y\in \cB$, the fibers $F_x$ and $F_y$ have same dimension. One defines then the {\em quotient vector bundle} $E/F$  as the bundle of base $\cB$ such that the fiber $\left(E/F\right)_x$ above $x$ is the set $\{e_x+F_x,e_x\in E_x\}$ of affine subspaces of $E_x$ directed by $F_x$. The bundle $F$ is endowed with the restricted metric $\|.\|_F$ and the norm of any element $e_x+F_x$ of $E/F$ is defined by the minimum of the norms of the vectors of  $e_x+F_x$. If $G$ is another subbundle of $E$, then one defines the vector subbundle $G/F\subset E/F$ as the image of $G$ by the canonical projection $E \to E/F$.

If $F$ is a subbundle invariant for the linear cocycle $\cA$ (that is, $\cA(F)=F$), 
 then $\cA$ induces canonically a {\em restricted cocycle} $\cA_{|F}$, and a {\em quotient cocycle} $\cA_{/F}$ defined on the quotient $E/F$ by $\cA_{/F}(e_x+F_x)=\cA(e_x+F_x)=\cA(e_x)+F_{\cA(x)}$. If $G$ is another invariant subbundle, then $G/F$ is an invariant subbundle for $\cA_{/F}$.
 
 \begin{remark}
 If $\cA$ is bounded by some constant $C>1$, then so are the restriction $\cA_{|F}$ and the quotient $\cA_{/F}$.
 \end{remark}

We use the natural notions of transverse subbundles and direct sum of transverse subbundles.
The following definition generalizes the definition given in the previous section for diffeomorphisms.
Let $\cA$ be a linear cocycle on a bundle $E$, and let $E=F\oplus G$ a splitting into two subbundles invariant by $\cA$. It is a  {\em dominated splitting} if and only if  there exists $N$ such that, for any point $x \in \cB$, for any unit vectors $u \in F_x$, $v\in G_x$ in the tangent fiber above $x$, we have $$\|\cA^N(u)\|<1/2.\|\cA^N(v)\|.$$
Given such $N$, one says that the splitting  $F\oplus G$ is {\em $N$-dominated}. The strength of a dominated splitting is given by the minimum of such $N$. The bigger that minimum, the weaker the domination. 

\subsection{Isotopic perturbation results on cocycles.}\label{s.pathcocycles}

A few perturbation results on cocycles are proved in~\cite{BGV} and~\cite{Gou}. Here we want to show that these perturbations can actually be reached through isotopies of cocycles that satisfy good properties, namely properties that will put us under the assumptions of Theorem~\ref{t.mainsimplestatement}.
 
To any tuple $(A_1,...,A_p)$ of matrices of $GL(d,\RR)$ one canonically associates the linear cocycle $\cA$ on the bundle $\cE=\{1,...,p\}\times \RR^d$ that sends the $i$-th fiber on the $(i+1)$-th fiber by the linear map of matrix $A_i$, and that sends the $p$-th fiber on the first by $A_p$.  The we say that $\cA$ is a {\em saddle cocycle} if and only if all the moduli of the eigenvalues of the product $A_p  ... A_1$ are different from $1$, and if there are some that are greater than $1$ and others that are less than $1$.  The splitting $\cE=E^s\oplus E^u$ into the stable bundle $E^s$ and the unstable one $E^u$ is called the {\em stable/unstable splitting}.

\medskip

Notice that Theorem~\ref{t.realeigen} is a straightforward consequence of Theorem~\ref{t.mainsimplestatement} and the following proposition about getting real eigenvalues:
\begin{proposition}\label{p.realeigen}
Let $\epsilon>0$, $C>1$ and $d\in \NN$. There exists an integer $N\in \NN$ such that,  for  any $p\geq N$ and any tuple $(A_1,...,A_p)$ of matrices in $GL(d,\RR)$, all bounded by $C$ (i.e. $\|A_i\|,\|A_i^{-1}\|<C$), it holds: 

\medskip

there is a path $\bigl\{\cA_{t}=(A_{1,t},\ldots,A_{p,t})\bigr\}_{t\in [0,1]}$ in $GL(d,\RR)^p$ such that
\begin{itemize}
\item $\cA_0=(A_1,...,A_p)$.
\item The radius of the path $\cA_t$ is less than $\epsilon$, that is,
$$\max_{1\leq n\leq p\atop t\in[0,1]}\left\{\|A_{n,t}-A_{n,0}\|, \|A^{-1}_{n,t}-A^{-1}_{n,0}\|\right\}<\epsilon.$$ 
\item  For all $t\in [0,1]$, the moduli of the eigenvalues of the product $B_t=A_{p,t}A_{p-1,t}...A_{1,t}$ (counted with multiplicity) coincide with the moduli of those of $B_0$ and the eigenvalues of $B_1$ are real.
\end{itemize}
 \end{proposition}

We state a second Proposition about reaching through an isotopy eigenvalues that all have moduli less than $\epsilon$ or more than $\epsilon^{-1}$. 

\begin{proposition}\label{p.hugeeigen}
Let $\epsilon>0$, $C>1$ and $d\in \NN$. There exists an integer $N\in \NN$ such that, for  any $p\geq N$ and any tuple $(A_1,...,A_p)$ of matrices in $GL(d,\RR)$, all bounded by $C$, if the moduli of the eigenvalues of the product $\prod A_k$ are pairwise distinct, then it holds: 

\medskip

there is a path $\bigl\{\cA_{t}=(A_{1,t},\ldots,A_{p,t})\bigr\}_{t\in [0,1]}$ in $GL(d,\RR)^p$ such that
\begin{itemize}
\item $\cA_0=(A_1,...,A_p)$.
\item The radius of the path $\cA_t$ is less than $\epsilon$, 
\item For all $t\in [0,1]$, the moduli of the eigenvalues of  $B_t=A_{p,t}...A_{1,t}$ are pairwise distinct and different from $1$ and the eigenvalues of $B_1$ have moduli less than $\epsilon$ or greater than $\epsilon^{-1}$.
\end{itemize}
 \end{proposition}

The third one is about obtaining a small angle in the absence of dominated splitting:

\begin{proposition}\label{p.smallangle}
Let $\epsilon>0$, $C>1$ and $d\in \NN$. There exists an integer $N\in \NN$ such that, for any $p\geq N$ and any tuple $(A_1,...,A_p)$ of matrices in $GL(d,\RR)$, all bounded by $C$, it holds:
\begin{itemize}
\item if the linear cocycle associated to it is a saddle cocycle such that its stable/unstable splitting is not $N$-dominated, 
\item if the eigenvalues of the product $A_p\times ...\times A_1$ are all real, 
\end{itemize}
there is a path $\bigl\{\cA_{t}=(A_{1,t},\ldots,A_{p,t})\bigr\}_{t\in [0,1]}$ in $GL(d,\RR)^p$ such that
\begin{itemize}
\item $\cA_0=(A_1,...,A_p)$.
\item The radius of the path $\cA_t$ is less than $\epsilon$, 
\item  For all $t\in [0,1]$, the eigenvalues of $B_t=A_{p,t}...A_{1,t}$ (counted with multiplicity) are equal to those of $B_0$.
\item The stable/unstable splitting of the cocycle associated to  $\cA_1$ has a minimum angle less than $\epsilon$.
\end{itemize}
\end{proposition}

\begin{proof}[Proof of Proposition~\ref{p.smallanglediffeo}:] Since it poses no difficulty, we only sketch it. One first applies Proposition~\ref{p.realeigen} to obtain a path that joins the cocycle corresponding to the derivative $Df_{|\Orb_P}$ along the orbit $\Orb_P$ of $P$ to a cocycle such that its eigenvalues are all real. Then adding an arbitrarily small path, one may suppose that the moduli of these eigenvalues are pairwise distinct. With Proposition~\ref{p.hugeeigen}, we prolong that path to obtain eigenvalues that have moduli less than $\epsilon$ or greater than $\epsilon^{-1}$. Remember that a weak dominated splitting remains a weak dominated splitting after perturbation, if it still exists. 
Hence, we can use Proposition~\ref{p.smallangle} to get a small angle. This provides us a path of small radius that joins the initial derivative to a cocycle that has all wanted properties. One finally applies Theorem~\ref{t.mainsimplestatementmoregeneral} to conclude the proof.
\end{proof}

\subsubsection{Proof of Proposition~\ref{p.realeigen}.} 

\begin{proof}[The dimension $d=2$ case: ]
First notice that, if the determinant of the product $A_{p}...A_{1}$ is negative, then the eigenvalues are already real and we are done. 

If not, one finds a $p$-periodic sequence of isometries $J_n$ of $\RR^2$, and a sequence of integers $C^{-1}\leq \lambda_n<C$, such that the matrix $\hat{A}_n=\lambda_n.J_nA_nJ_{n+1}^{-1}$ has determinant $1$. Notice that the product $\hat{A}_{p}...\hat{A}_{1}$ has real eigenvalues if and only if the product $A_{p}...A_{1}$ has real eigenvalues. 

Assume we have a path $\hat{\cA}_t=(\hat{A}_{1,t},...,\hat{A}_{p,t})$ of diameter less than $\hat{\epsilon}=C^{-1}\epsilon$, such that it holds
\begin{itemize}
\item  $\hat{\cA}_0=(\hat{A}_1,...,\hat{A}_p)$,
\item for all $t\in [0,1]$, the moduli of the eigenvalues of the product $\hat{B}_t=\hat{A}_{p,t}\hat{A}_{p-1,t}...\hat{A}_{1,t}$ coincide with the moduli of those of $\hat{B}_0$,
\item the eigenvalues of $\hat{B}_1$ are real.
\end{itemize}
Then the path $\cA_t=(A_{1,t},...,A_{p,t})$, where $A_{n,t}=\lambda_n^{-1}.J_n^{-1}\hat{A}_{n,t}J_{n+1}$, clearly satisfies all the conclusions of \cref{p.realeigen}. 
Therefore, it is enough to solve \cref{p.realeigen} for the $A_n\in SL(2,\RR)$ case. \cite[lemme 6.6]{BoCro} easily answers that case:

\begin{lemma}[Bonatti, Crovisier]\label{l.dim2}
For any $\varepsilon>0$, there exists $N(\varepsilon) \geq 1$ such that, for any integer $p \geq N(\varepsilon)$ and any finite sequence $A_1,...,A_p$ of elements in $SL(2,\RR)$, there exists a sequence $\alpha_1,...,\alpha_p$ in $]-\varepsilon,\varepsilon[$ such that the following assertion holds:

for any $i \in \{1,...,p\}$ if we denote by $B_i=R_{\alpha_i} \circ A_i$ the composition of $A_i$ with the rotation $R_{\alpha_i}$ of angle $\alpha_i$, then the matrix $B_p \circ B_{p-1} \circ \cdots \circ B_1$ has real eigenvalues.
\end{lemma}

Under the hypothesis of the lemma, let $\alpha_1,\ldots,\alpha_p$ be a corresponding sequence. For all $1\leq i\leq p$, define $A_{t,i}=R_{t.\alpha_i} \circ A_i$, and let $t_0$ be the least positive number such that the matrix $A_{t,p}  \circ \cdots \circ A_{t,1}$ has real eigenvalues. Then the path  $\left\{(A_{t,1},...,A_{t,p})\right\}_{t\in [0,t_0]}$ satisfies the conclusions of Proposition~\ref{p.realeigen}. 

This ends the proof of the dimension $2$ case. \end{proof}

\begin{proof}[Proof of Proposition~\ref{p.realeigen} in any dimensions] Consider the linear cocycle $\cA$ associated to the sequence $A_1,...,A_p$ on the bundle $\cE=\{1,...,p\}\times \RR^d$.
 If some eigenvalue of the product $A_p\ldots A_1$, that is the first return map, is not real, there is a dimension $2$ invariant subbundle $F$ of $\cE$ that corresponds to the corresponding pair of complex conjugated eigenvalues. Choosing orthonormal basis in each fibre of $F$ and completing by a basis of the orthonormal bundle $F^\perp$, the linear cocycle $\cA$ writes in those bases as a sequence of matrices of the form:
$$
\left(
\begin{array}{cc}
A_{|F,i}& B\\
0& A^\perp_{F,i}
\end{array}
\right).
$$
Using the proposition in dimension $2$, one may choose a path $\cA_{|F,t}$ of automorphisms  of $\cF$ ending at $\cA_{|F}$ such that the first return map of $\cA_{|F,0}$ has real eigenvalues. Denote by $\cA_t$ the linear cocycle corresponding to the sequences of the matrices
$$
\left(
\begin{array}{cc}
A_{|F,t,i}& B\\
0& A^\perp_{F,i}
\end{array}
\right).
$$
This defines a path of small radius that joins the initial automorphism to an automorphism where two of the eigenvalues have turned real. The other eigenvalues are given by the product of the blocks $A^\perp_{F,i} $, therefore did not change.
One may need to iterate that process at most $d/2$ times to turn all eigenvalues real, by concatenation of small paths. This ends the proof of the proposition. 
\end{proof}

\subsubsection{Proof of Proposition~\ref{p.hugeeigen}.}
As in the previous proof, one considers the linear cocycle $\cA$ associated to the sequence $A_1,...,A_p$ on the bundle $\cE=\{1,...,p\}\times \RR^d$. Let $\cE=E^s\oplus E^u$ be the stable/unstable splitting for the cocycle $\cA$. Choosing an orthonormal basis in each fibre of $E^s$ and completing by a basis of the orthonormal bundle $E^{s\perp}$, the linear cocycle $\cA$ writes in those bases as a sequence of matrices of the form:
$$
\left(
\begin{array}{cc}
A_{|E^s,i}& B\\
0& A^\perp_{E^s,i}
\end{array}
\right).
$$
Let $0<t\leq1$. Let $\cA_t$ be the cocycle obtained from $\cA$  multiplying each matrix $A_{|E^s,i}$ by $t^{1/p}$. One easily checks that the stable eigenvalues of $\cA_t$ are those of  $\cA$ multiplied by $t$, while the unstable eigenvalues remain unchanged. All stable eigenvalues for $\cA_{\epsilon}$ are less than $\epsilon$ and, for $p$ big, the path $\{\cA_t\}_{t\in[\epsilon,1]}$ is small. One can do the same for the unstable eigenvalues of $\cA_{\epsilon}$ and obtain another path. The concatenation of both paths ends the proof of the proposition.

\subsubsection{Proof of Proposition~\ref{p.smallangle}.}
We show it by induction on the dimension $d$. We first restate ~\cite[Lemma 4.4]{BDP}:

\begin{lemma}[Bonatti, D\'iaz, Pujals] \label{l.bdp4.4}
Let $C>1$ and $d\in \RR$. There exists a mapping $\phi_{C,d}\colon \NN \to \NN$ such that, for any linear cocycle $\cA$ bounded by $C$ on a $d$-dimensional bundle $E$, the following holds for all $N\in \NN$: if an invariant splitting $E=F\oplus G$ is not $\phi_{C,d}(N)$-dominated for $\cA$, and if $H\subset F$ (resp. $H\subset G$) is an invariant subbundle, then 

\begin{itemize}
\item either the splitting $H\oplus G$ (resp. $F\oplus H$) is not $N$-dominated for the restriction $\cA_{|H\oplus G}$ (resp. $\cA_{|F \oplus H})$,
\item or $F/H\oplus G/H$ is not  $N$-dominated for the quotient $\cA_{/H}$.
\end{itemize}
\end{lemma}

\begin{proof}[Proof in dimension 2:] This is basically~\cite[Lemma~7.10]{BDV} by isotopy. Notice that the perturbations done in the proof of that lemma can be obtained by an isotopy such that at each time, two invariant bundles exist. The eigenvalues may be slightly modified along that isotopy, however each eigenvalue may be retrieved by dilating or contracting normally to the other eigendirection (which preserves the other eigenvalue).
\end{proof}

\begin{proof}[Proof in any dimension:]
Fix $d>2$, and assume that the proposition in proved in all dimensions less than $d$. Let $C>1$ and $\cA$ be a saddle cocycle bounded by $C$ associated  to a sequence $A_1,...,A_p$ on the bundle $\cE=\{1,...,p\}\times \RR^d$ and let $\cE=E^s\oplus E^u$ be the stable/unstable splitting. One of these two bundles has dimension greater or equal to $2$, we assume it is $E^s$ (the other case is symmetrical).
Since the eigenvalues of $\cA$ are real, there is a proper invariant subbundle $F\subset E^s$. 
For all $N\in \NN$, if the stable/unstable splitting $E^s\oplus E^u$ is not $\phi_{C,d}(N)$-dominated, by Lemma~\ref{l.bdp4.4}, either $H=F\oplus E^u$ is not $N$-dominated for the restriction $\cA_{|H}$, or $E^s/F\oplus E^u/F$ is not $N$-dominated for $\cA_{/F}$. 
Let $\epsilon>0$. By the induction hypothesis, one can find $N_{d'}\in \NN$ such that the conclusions of \cref{p.smallangle} are satisfied with respect to $\epsilon, C$ and any $2\leq d'<d$. 

Notice that for any $N$ greater than some $\tilde{N}_{d'}$ it holds: if a $d'$-dimensionnal saddle cocycle is bounded by $C$ and not $N$-dominated, then it is not $N_{d'}$-dominated. Let $$N_0=\max_{2\leq d'<d}\{\tilde{N}_{d'}\}.$$ 
%For all $2\leq d'<d$, if a $d'$-dimensionnal saddle cocycle is bounded by $C$ and not $N_0$-dominated, then it is not $N_{d'}$-dominated. 

Then, if $\cA$ is not $\phi_{C,d}(N_0)$-dominated, by \cref{l.bdp4.4} and the induction hypothesis, one has either:

\begin{itemize}
\item a path $\cA_{|H,t}$ of radius  $\leq\epsilon$ that joins $\cA_{|H}$ to a saddle cocycle that has a minimum stable/unstable angle less than $\epsilon$, and such that the eigenvalues are preserved all along the path. One may extend that path to a path $\cA_t$ of saddle cocycles on $\cE$, the same way as we extended the path $\cA_{|F}$ in the proof of Proposition~\ref{p.realeigen}. That extended path has the same radius as $\cA_{|H,t}$. The minimum stable/unstable angle of $\cA_t$  is less or equal to that of $\cA_{|H,t}$, for all $t$, in particular that of $\cA_1$ is less than $\epsilon$. Finally, for all $t$, the eigenvalues of  $\cA_t$ are the same as those of $\cA$.   

\medskip
\item or a path $\cA_{/F,t}$  of radius $\leq\epsilon$ that joins $\cA_{/F}$ to a saddle cocycle that has a minimum stable/unstable angle less than $\epsilon$, and such that the eigenvalues of the first return map are preserved all along the path. Choosing an orthonormal basis in each fibre of $F$ and completing by a basis of the orthonormal bundle $F^\perp$, the linear automorphism $\cA$ writes in those bases as a sequence of matrices of the form:
$$
\left(
\begin{array}{cc}
A_{|F,i}& B\\
0& A^\perp_{F,i}
\end{array}
\right),
$$
where the sequence of matrices $A^\perp_{F,i}$ identifies with the quotient $\cA_{/F}$. We define a path $\cA_t$ replacing the sequence $A^\perp_{F,i}$ by the sequence $A^\perp_{F,t,i}$ that corresponds to the cocycle $\cA_{/F,t}$. As both $\cA_{|F}$ and $\cA_{/F,t}$ are saddle cocycles, for all $t$, $\cA_t$ is also a saddle cocycle. 

Let $\cE=E^s_t\oplus E^u_t$ be the stable/unstable splitting for $\cA_t$. By construction $F$ is a subbundle of $E^s_t$ and is invariant by $\cA_t$. The stable/unstable splitting of ${\cA_{t}}_{/F}$, which identifies to $\cA_{/F,t}$, is $\cE/F={E^s_t}/F\oplus {E^u_t}/F$. Notice that, given three vector subspaces $\Gamma \subset \Delta$ and $\Lambda$ of $\RR^d$, one has the following relation on minimum angles: $$\angle(\Delta,\Lambda)\leq \angle(\Delta/\Gamma,\Lambda/\Gamma).$$ 
Therefore, the minimum stable/unstable angle of each $\cA_t$ is less than that of $\cA_{/F,t}$, in particular, that of $\cA_1$ is less than $\epsilon$. The path $\cA_t$ has same radius as the quotient path  $\cA_{/F,t}$, in particular it is less than $\epsilon$. The eigenvalues are the same for $\cA=\cA_0$ and $\cA_t$.
\end{itemize}
We are done in both case, which ends the proof of Proposition~\ref{p.smallangle}.
\end{proof}

\section{Further results and announcements}\label{s.furtherresults}

In this paper, we assume that some $i$-strong stable/unstable directions exist at any time $t$ of the homotopy, and we obtain a perturbation lemma that preserves the corresponding local invariant manifolds entirely, outside small neighbourhoods.

We announce a 'manifolds prescribing pathwise Franks Lemma', that is, a generalisation of Theorem~\ref{t.mainsimplestatement} that allows to prescribe the strong stable/unstable manifolds within any 'admissible' flag of stable/unstable manifolds. That generalisation implies for instance that if the $i$-strong stable direction exists for all the cocycles  $\gamma_t$, for $0\leq t\leq 1$, and if, for some time $t_0$, all the eigenvalues inside the $i$-strong stable direction have same moduli, then one can do the pathwise Franks' lemma, prescribing the $j$-strong stable manifolds, for all $j\leq i$, inside arbitrarily large annuli of fundamental domains of $i$-strong stable manifold.

Let us formally define these objects. Let $f$ be a $C^1$-diffeomorphism and $P$ be a periodic saddle point for $f$. To simplify the statement, we assume that $P$ is a fixed point. Given a fundamental domain of the stable/unstable manifold of $P$ identified diffeomorphically to $\SS^{i_s-1}\times[0,1[$, an {\em annulus} $A(f,P)$ is a subset of the form $\SS^{i_s-1}\times[0,\rho[$, where $0<\rho<1$. We denote by $W^{s,i}(f)$ the $i$-strong stable manifold of $f$.
An {\em $i$-admissible flag of manifolds for $f$} is a flag $W^{s,1}\subset ... \subset W^{s,i}=W^{s,i}(f)$ of $f$-invariant manifolds such that each $W^{s,k}$ is an immersed boundaryless $k$-dimensional manifold that contains $P$, and that is smooth at all points, but possibly $P$. A particular case (and simple case) of the announced Franks'  Lemma that prescribes manifolds can be stated as follows:

\begin{othertheorem}
Assume that $(\cA_t)_{t\in[0,1]}$ is a path that starts at the sequence of matrices $\cA_0$ corresponding to the derivative of $f$. Assume that, for all $t$, the corresponding first return map has an $i$-strong stable direction. Assume also that there is some time $t_0$ such that the $i$ strongest stable eigenvalues $\lambda_1(t_0),...\lambda_i(t_0)$ of $\cA_{t_0}$, counted with multiplicity, have same moduli. Then, for any $i$-admissible flag $W^{s,1}\subset ... \subset W^{s,i}$ for $f$, for any annulus $A(f,P)$, for any neighbourhood $U$ of the orbit of $P$, there is a diffeomorphism $g$ such that it holds:
\begin{itemize}
\item $g$ is a perturbation of $f$ whose size can be taken arbitrarily close to the radius of the path $\cA_t$,
\item $g^{\pm 1}=f^{\pm 1}$ on the orbit $\Orb_P$ of $P$ and outside $U$,
\item the sequence of matrices $\cA_1$ corresponds to the derivative $Dg_{|\Orb_P}$,
\item for all $1\leq j\leq i$, if $g$ has a $j$-strong stable manifold, then it coincides with $\cW^{s,j}$ by restriction to the annulus $A(f,P)$.
\end{itemize}
\end{othertheorem}

The perturbation techniques for linear cocycles as developed in~\cite{M1,BDP,BGV} successively, can be easily rewritten in order to take into account the need of a good path between the initial cocycle and the pertubation. The perturbations of cocycles obtained by the techniques of~\cite{BGV} can indeed be done along paths whose size are small (R. Potrie actually wrote a proof of it in~\cite{Po}). A general description of the vectors of Lyapunov vectors that can be reached by a perturbation of a linear cocycle has been recently given by Bochi and Bonatti~\cite{BoBo}; moreover, those perturbations are built so that they can be reached from the initial cocycle by a isotopy. 
These isotopic perturbation lemmas for cocycles and the theorem announced above lead to easy and systematic ways to create strong connections and heterodimensional cycles whenever there is some lack of domination within a homoclinic class.

We claim that with some hypotheses on the signs of the eigenvalues of the first return map of $\cA_1$, the theorem above can be adapted to prescribe the entire semi-local flag of strong stable manifolds outside $U$ within an isotopy class of $i$-admissible flags determined by the isotopy class of the path of eigenvalues $\bigl(\lambda_1(t),...\lambda_i(t)\bigr)$ (here $\lambda_j(t)$ is the $j$-th eigenvalue of $\cA_t$, counting with multiplicity).
In a work in progress, Bonatti and Shinohara used an adapted version of this argument in dimension $2$, in order to build their new examples of wild $C^1$-generic dynamics.

Finally, we claim that these results, with some more work and excluding the codimension one manifolds\footnote{it is possible to preserve {\em annuli} of codimension $1$ stable or unstable manifolds by conservative perturbations, however there seems to be an obstruction to preserving them semi-locally}, can be adapted to hold in the volume preserving and symplectic settings. They can also clearly be adapted to the flows case, but here again technical work is needed.

\end{document}